\documentclass{amsart}

\usepackage[margin=1in]{geometry}
\usepackage{amsmath,amsthm,amssymb,array,enumerate,parskip,graphicx,tabularx}
\usepackage[all]{xy}
\usepackage{graphicx}

\usepackage{xcolor} 
\usepackage{mathrsfs}

\usepackage[pagebackref]{hyperref}
\definecolor{dark-red}{rgb}{0.5,0.15,0.15}
\definecolor{dark-blue}{rgb}{0.15,0.15,0.6}
\definecolor{dark-green}{rgb}{0.15,0.6,0.15}
\hypersetup{
    colorlinks, linkcolor=dark-red,
    citecolor=dark-blue, urlcolor=dark-green
}

\usepackage[capitalise]{cleveref}
\usepackage{comment}

\newcommand{\Z}{\mathbb{Z}} 
\newcommand{\Q}{\mathbb{Q}} 

\newcommand \rH {{\rm H}}

\newcommand \epz {[\epsilon]_0}
\newcommand \epo {[\epsilon]_1}
\newcommand \epzpi {\epsilon_0}
\newcommand \epopi {\epsilon_1}

\newcommand \ZZ {{\mathbb Z}}

\newcommand \QQ {{\mathbb Q}}

\newtheorem{theorem}{Theorem}[section]

\newtheorem{lemma}[theorem]{Lemma} 
\newtheorem{corollary}[theorem]{Corollary}
\newtheorem{example}[theorem]{Example}
\newtheorem{proposition}[theorem]{Proposition} \theoremstyle{definition}
\newtheorem{remark}[theorem]{Remark} 
\newtheorem{definition}[theorem]{Definition}

\title[Galois action on the lower central series of the fundamental group of the Fermat curve]{The Galois action on the lower central series of the fundamental group of the Fermat curve}
\author{Rachel Davis}
\address{University of Wisconsin-Madison}
\email{rachel.davis@wisc.edu}
\author{Rachel Pries}
\address{Colorado State University}
\email{pries@colostate.edu}
\author{Kirsten Wickelgren}
\address{Duke University}
\email{kirsten.wickelgren@duke.edu}

\thanks{We would like to thank AIM for support for this project through a Square collaboration grant.
We would like to thank Vesna Stojanoska for earlier collaboration and Richard Hain 
for helpful comments.  We would like to thank the anonymous referee for thoughtful comments. 
Pries was supported by NSF grants DMS-15-02227 and DMS-19-01819. Wickelgren was supported by an American Institute of Mathematics five year fellowship and NSF grants DMS-1406380, DMS-1552730, and DMS-2001890. }

\begin{document}

\begin{abstract}
Information about the absolute Galois group $G_K$ of a number field $K$ is encoded in how it acts on the
\'etale fundamental group $\pi$ of a curve $X$ defined over $K$.
In the case that $K=\QQ(\zeta_n)$ is the cyclotomic field and $X$ is the Fermat curve of degree $n \geq 3$, 
Anderson determined the action of $G_K$ on the \'etale homology with coefficients in $\ZZ/n \ZZ$.
The \'etale homology is the first quotient in the lower central series of the \'etale fundamental group.
In this paper, we determine the Galois module structure of the graded Lie algebra for $\pi$.
As a consequence, this determines the action of $G_K$ on all degrees of the associated graded quotient of the lower central series
of the \'etale fundamental group of the Fermat curve of degree $n$, with coefficients in $\ZZ/n \ZZ$.\\
MSC2010: 11D41, 11R18, 13A50, 14F20, 14F35, 14H30\\
Keywords: Fermat curve, cyclotomic field, \'etale fundamental group, homology, lower central series, 2-nilpotent quotient, Galois module.
\end{abstract}


\maketitle

\section{Introduction}

Let $X$ be the Fermat curve of degree $n$, where $n \geq 3$.
Consider the cyclotomic field $K = \QQ(\zeta_n)$; let $\overline{K}$ be its algebraic closure, 
and let $G_K$ be its absolute Galois group.
Anderson described the action of $G_K$ on the \'etale homology $\rH_1(X; \Z/n\Z)$ with coefficients in $\Z/n \Z$ 
of the base change $X_{\overline{K}}$ of $X$ to $\overline{K}$ 
(the base change is suppressed in the notation $\rH_1(X; \Z/n\Z)$); 
more precisely, he analyzed the $G_K$-action on the relative homology $\rH_1(U,Y; \Z/n\Z)$ 
of the open affine Fermat curve $U = \{ (x,y): x^n + y^n = 1 \}$ relative to the set $Y$ of the $2n$ cusps with $xy=0$. 

The main result of \cite[Sections 4 \& 5]{WINF:birs} is that Anderson's description uniquely determines the 
action of $G_K$ on $\rH_1(U,Y; \Z/n\Z)$ when $n$ is prime.
In \cite[Theorem 1.1]{Baction}, the authors find an explicit formula for the action of each $\sigma \in G_K$
on $\rH_1(U,Y; \Z/n\Z)$ when $n$ is a prime satisfying Vandiver's conjecture. 

Let $\pi = [\pi]_1 = \pi_1(X)$ be the \'etale fundamental group of $X_{\overline{K}}$, and for $m \geq 2$, let $[\pi]_m$ be the $m$th subgroup of the lower central series $\pi = [\pi]_1 \supset [\pi]_2 \supset [\pi]_3 \supset \cdots$, defined so that $[\pi]_m=\overline{[\pi, [\pi]_{m-1}]}$ is the closure of the subgroup generated by 
commutators of elements of $\pi$ with elements of $[\pi]_{m-1}$. For example, there is a canonical isomorphism of $G_K$-modules $\rH_1(X; \Z/n\Z) \cong \pi/[\pi]_2 \otimes \Z/n \Z$, and as a group $\pi/[\pi]_2 \otimes \Z/n \Z \cong (\Z/n\Z)^{2g}$, where $g=(n-1)(n-2)/2$ is the genus of the Fermat curve.

In this paper, we describe the action of $G_K$ on each of the higher graded quotients 
$[\pi]_m/[\pi]_{m+1} \otimes \Z/n \Z$ in the lower central series filtration of $\pi$, with coefficients in $\Z/n \Z$; when $n$ is prime, this description determines the action uniquely.
One motivation for this work is that it sheds light on the $2$-nilpotent quotient of the \'etale fundamental group of the 
Fermat curve, because of the exact sequence:
\begin{equation}\label{2nilpiX}
1 \to [\pi]_2/[\pi]_3 \otimes \Z/n \Z\to \pi/[\pi]_3 \otimes \Z/n \Z \to \pi/[\pi]_2 \otimes \Z/n \Z \to 1.
\end{equation}

To state the results more precisely,
consider the graded Lie algebra ${\rm gr}(\pi) = \oplus_{m \geq 1} [\pi]_m/[\pi]_{m+1}$ 
associated with the lower central series for $\pi$, (\cite{lazard, serre65}), which is equipped with its $G_K$-action. The group $\mu_n \times \mu_n$ acts on $X$ by multiplying $x$ and $y$ by $n$th roots of unity, and therefore acts 
$G_K$-equivariantly on $\pi$. 
Let $F$ be the free profinite group on $2g$ generators, and consider its graded Lie algebra 
${\rm gr}(F)= \oplus_{m \geq 1} {\rm gr}_m(F)$. 
It follows from work of Labute in \cite[Theorem, page 17]{labute70} that 
there is an element $\rho$ of weight $2$ such that 
\[{\rm gr}(\pi)  \cong {\rm gr}(F)/\overline{\langle \rho \rangle}  .\] The right hand side may be equipped with a Galois action by identifying ${\rm gr}(F)$ with the \'etale fundamental group of the open complement in $X_{\overline{K}}$ of a $K$-rational point, and the isomorphism may be chosen to be the one induced from the inclusion of the open subscheme. It thus respects the Galois actions on both sides. This Galois action is determined by the action on ${\rm gr}_1(F)\cong {\rm gr}_1(\pi)\cong \rH_1(X; \hat{\Z})$ by \cite[Section 5.7, Corollary~5.12~(v)]{Magnus-Karrass-Solitar}. 

In Theorem~\ref{Tintro}, we determine the isomorphism class of ${\rm gr}(\pi)$ as a graded Lie group with action of 
$\mu_n \times \mu_n$.  Since ${\rm gr}(F)$ is generated in degree $1$, 
it suffices to obtain a complete description of the ideal $\overline{\langle \rho \rangle} \subset {\rm gr}(F)$ and the action 
of $\mu_n \times \mu_n$ on it.

Furthermore, when $n$ is prime, we determine the isomorphism class of
${\rm gr}(\pi) \otimes \Z/n\Z$ as a graded Lie algebra 
with $\mu_n \times \mu_n\times G_K$-action.  This gives the stated application of  
determining the action of $G_K$ on each of the higher graded quotients 
$[\pi]_m/[\pi]_{m+1} \otimes \Z/n \Z$.
For this, it suffices to use the description of the ideal $\overline{\langle \rho \rangle} \subset {\rm gr}(F)$
from Theorem \ref{Tintro} together with the action of $G_K$ on $[\pi]_1/[\pi]_{2}  \otimes \Z/n \Z$
from our earlier result in \cite[Theorem 1.1]{Baction}.

To find the ideal $\overline{\langle \rho \rangle} \subset {\rm gr}(F)$, 
we use the isomorphism of $G_K$-modules \cite[Corollary 8.3]{hain}
\[\displaystyle [\pi]_2/[\pi]_3 \cong \left(\rH_1(X) \wedge \rH_1(X)\right)/{\rm Im}(\mathscr{C}),\]
where 
\begin{equation}\label{defmapC}
{\mathscr{C}}: \rH_2(X) \to \rH_1(X) \wedge \rH_1(X)
\end{equation}
is the dual map to the cup product $\rH^1(X) \wedge \rH^1(X) \to \rH^2(X)$.

The image ${\rm Im}(\mathscr{C})$ is cyclic since $\rH_2(X) \cong \Z(1)$. 
We use the basis of $\rH_1(X)$ as a $\Z$-module from \cite[Theorem 1.2]{E:integral}, see \eqref{Ehomproj}, 
which interacts well with the $\mu_n \times \mu_n$-action.  
This basis gives an isomorphism ${\rm gr}_1(F)  \cong \rH_1(X)$, which in turn induces an isomorphism 
${\rm gr}_2(F) \cong \rH_1(X) \wedge \rH_1(X)$. We may therefore compute $\rho$ as an element of 
$\rH_1(X) \wedge \rH_1(X)$, and any generator of ${\rm Im}(\mathscr{C})$ is a valid choice for $\rho$. 

We note that $\rH_1(X)$ is a quotient of $\rH_1(U)$, which is a subspace 
of the relative homology $\rH_1(U, Y)$.  
For all integers $n \geq 3$, we determine $\rho$ as the image of an element $\Delta$ in $\rH_1(U) \wedge \rH_1(U)$, 
with the result expressed in terms of a basis $\{[E_{i,j}]\}$ for $\rH_1(U)$
defined in Section \ref{subsectionwithEijdef}, Lemma \ref{LfactE}.
This basis is convenient because we know the action of $\mu_n\times \mu_n$ and $G_K$ on its elements.

\begin{theorem} \label{Tintro}
For $n \geq 3$, a generator $\rho$ for ${\rm Im}(\mathscr{C})$
is given by the image in $\rH_1(X) \wedge \rH_1(X)$
of the following element $\Delta$ of $\rH_1(U) \wedge \rH_1(U)$:
\[\displaystyle \Delta = 
\sum_{\substack{1 \leq i_1 \leq i_2 \leq n-1 \\ 1 \leq j_1, j_2 \leq n-1 \\
(i_1, j_1) \not = (i_2, j_2)} } \epsilon(i_1,j_1,i_2,j_2) [E_{i_1,j_1}] \wedge [E_{i_2,j_2}],\]
where
\[\epsilon(i_1,j_1,i_2,j_2) =\begin{cases} 
      1 & \text{if } j_2-j_1 \equiv i_2-i_1 \not \equiv 0 \mod n-1  \\
     -1 & \text{if } j_2-j_1 +1 \equiv i_2-i_1 \not \equiv 0 \mod n-1 \\
      0 & \text{otherwise.}
   \end{cases}
\]
\end{theorem}

The action of the absolute Galois group $G_K$ on the homology of the Fermat curve is the subject of 
several foundational papers, including \cite{Ihara}, \cite{Anderson}, \cite{AI-prol}, \cite{Anderson_hyper}, and \cite{Coleman89}.
Let $n=p$ be an odd prime.  Let $L$ be the splitting field of $1-(1-x^p)^p$.
In \cite[Section 10.5]{Anderson}, 
Anderson proved that the action of $G_K$ on the relative homology 
$\rH_1(U,Y; \Z/p\Z)$ factors through the finite Galois extension $L/K$
and gave a theoretical formulation for the action of $q \in Q ={\rm Gal}(L/K)$.
From \cite[Section 10.5]{Anderson}
and the result of Labute quoted above, it follows that
the action of $G_K$ on $[\pi]_m/[\pi]_{m+1} \otimes \Z/p \Z$
factors through $Q ={\rm Gal}(L/K)$, for any $m \geq 2$. 

In \cite[Theorem 1.1]{Baction} and \cite[Theorem 1.1]{WINF:birs}, we made a completely explicit calculation of the $Q$-action on $\rH_1(U, Y; \ZZ/p \ZZ)$
when $p$ is an odd prime satisfying Vandiver's conjecture.\footnote{Vandiver's Conjecture states that $p$ does not divide the order of the class group of $\QQ(\zeta_p+ \zeta_p^{-1})$.  
It is true for all regular primes $p$ and all primes less than $163$ million.}
Our main motivation for Theorem \ref{Tintro} is that the $G_K$-module $[\pi]_2/[\pi]_3$ occurs as the coefficient
group in a map that measures an obstruction for rational points:
$\delta_2: \rH^1(G_K, \rH_1(X)) \to \rH^2(G_K, [\pi]_2/[\pi]_3)$.
For this reason, we highlight the following result.

\begin{corollary} \label{Cintro}
Combining \cite[Theorem 1.1]{Baction} with Theorem \ref{Tintro} yields an explicit computation of $[\pi]_2/[\pi]_3 \otimes \Z/p \Z$ as a $G_K$-module when $p$ is an odd prime satisfying Vandiver's conjecture.
\end{corollary}

Section \ref{Sexample} contains several applications.
In Section \ref{Sverify}, we give an independent verification for the formula for $\rho$ if $p=5$, using 
the fact that $\rho$ satisfies certain invariance properties under the action of ${\rm Aut}(X)$ and ${\rm Gal}(L/\QQ)$.
Using these invariance properties, if $p=5$, we also compute that the dimension of the 
$G_\QQ$-invariant subspace of $\rH_1(X; \ZZ/5\ZZ)$ is $2$, see Example \ref{EGQinv}.
This provides a new proof of a result of Tzermias \cite[Corollary 2]{tzermias5}.

In Section \ref{Sapplication}, we consider the short exact sequence
\[0 \to (\ZZ/p\ZZ)\rho \to \rH_1(X; \ZZ/p\ZZ) \wedge \rH_1(X; \ZZ/p\ZZ) \to [\pi]_2/[\pi]_3 \otimes \ZZ/p \ZZ \to 0.\]
Since $Q$ fixes $\rho$, this yields a long exact sequence
\begin{equation} \label{Elongexactintro}
0 \to (\ZZ/p\ZZ) \rho \to \rH^0(Q; \rH_1(X; \ZZ/p\ZZ) \wedge \rH_1(X; \ZZ/p\ZZ)) 
\to \rH^0(Q; [\pi]_2/[\pi]_3 \otimes \ZZ/p\ZZ) \stackrel{\delta}{\to} \rH^1(Q; (\ZZ/p\ZZ) \rho).
\end{equation}
If $p=5$, as an application of Corollary \ref{Cintro}, we compute that the dimension of the
$G_K$-invariant subspace of $\rH_1(X; \ZZ/5\ZZ) \wedge \rH_1(X; \ZZ/5\ZZ)$ 
(resp.\ $[\pi]_2/[\pi]_3 \otimes \Z/5 \Z$) is $35$ (resp.\ 34).
This shows that the coboundary map $\delta$ in \eqref{Elongexactintro} is trivial if $p=5$, 
see Example \ref{Aqinv}; this is a non-trivial fact since $p \mid |Q|$.

\section{The fundamental group of the Fermat curve} \label{S2}

Let $\zeta=\zeta_n = e^{2 \pi i/n}$ (resp.\ $\epsilon=e^{\pi i/n}$) be a primitive $n$th (resp. $2n$th) root of unity.

Consider the Fermat curve $X$ of exponent $n$ with equation $x^n+y^n=z^n$.
Let $Z_0$ be the set of $n$ points where $z=0$.  
Consider the open affine subset $U = X-Z_0$.  
In Sections \ref{S2}-\ref{S4}, the field of definition is the complex numbers; 
let $X:=X({\mathbb C})$ and 
$U:=U({\mathbb C}) = \{(x,y) \in {\mathbb C}^2 \mid x^n + y^n =1\}$.

\subsection{The fundamental group and the definition of $\Delta$} \label{Sfund}

Note that $U$ is a real surface of genus $g={n-1 \choose 2}$ with $n$ punctures. 
We choose the base point $b=(0,1)$.
There exist loops $a_i, z_i$ for $1 \leq i \leq g$ and $c_j$ for $0 \leq j \leq n-1$, with base point $b$, such that
$\pi_1(U)$ has a presentation 
\begin{equation}\label{pi1Upresentation}\pi_1(U) = \langle a_i, z_i, c_j: i=1,\ldots,g,~~j=0,\ldots,n-1\rangle/ \prod_{i=1}^g [a_i,z_i] \prod_{j=0}^{n-1} c_j.\end{equation}
Let $\bar{a}_i, \bar{z}_i, \bar{c}_j$ denote the images of $a_i, z_i, c_j$ in $\rH_1(U)$. $\rH_1(U)$ is equipped with an intersection pairing, which may be defined using Poincar\'e duality $\rH_1(U) \cong \rH^1_c(U)$ and the cup product on compactly supported cohomology.
We can suppose that
\begin{enumerate}
\item \label{en:cj_around_zetaj} the loop $c_j$ circles the puncture $[\zeta^j:\epsilon: 0] \in Z_0$; 
\item \label{en:cj_triv_pairing} each $\bar{c}_j$ pairs trivially with $\bar{a}_i,\bar{z}_i$; each $\bar{c}_j$ pairs trivially with 
$\bar{c}_i$ for $i \not = j$;
\item \label{en:az_symp} and the images of $\bar{a}_i, \bar{z}_i$ in $\rH_1(X)$ form a standard symplectic basis.
\end{enumerate} 

Item \eqref{en:cj_triv_pairing} may be arranged by choosing loops $c_j$ whose images have no set-theoretic intersection with each other or with the loops $a_i, z_i$ for $1 \leq i \leq g$. 
This in turn can be arranged using a standard gluing of an $n$-punctured
polygon with $4g$ sides, with the sides labeled consecutively by 
$a_1, z_1, a_1^{-1}, z_1^{-1}, \ldots, a_g, z_g, a_g^{-1}, z_g^{-1}$.

The set $\{\bar{c}_j\}$ generates the kernel of $\rH_1(U) \to \rH_1(X)$.
Define 
\begin{equation} \label{Edelta}
\Delta = \sum_{i=1}^{g} \bar{a}_i \wedge \bar{z}_i \in \rH_1(U) \wedge \rH_1(U).
\end{equation}

Let $[\pi_1(U)]_2=\overline{[\pi_1(U), \pi_1(U)]}$ and $[\pi_1(U)]_3=\overline{[\pi_1(U), [\pi_1(U)]_2]}$.
Consider the map 
\[C: \rH_1(U) \wedge \rH_1(U) \to [\pi_1(U)]_2/[\pi_1(U)]_3,\] which takes 
the simple wedge $r \wedge s$ to the (equivalance class of the) commutator of a lift of $r$ and a lift 
of $s$ to elements of $\pi_1(U)$.
Since $U$ is not proper, $C$ is an isomorphism.
Recall the definition of ${\mathscr C}$ from \eqref{defmapC}.

\begin{lemma} \label{Ltworho}
The image of $\Delta$ under the map $\wedge^2(\rH_1(U) \to \rH_1(X))$ is a generator 
of ${\rm Im}({\mathscr C})$.
\end{lemma}

\begin{proof}
By definition, $\mathscr{C}$ is the dual of the cup product pairing.  
Since the images of $\bar{a}_i, \bar{z}_i$ form a standard symplectic basis,
the image of $\mathscr{C}$ is generated by the image of $\sum_{i=1}^{g} \bar{a}_i \wedge \bar{z}_i$
under the map $\wedge^2(\rH_1(U) \to \rH_1(X))$, which is 
the image of $\Delta$ by definition.
\end{proof}

Lemma \ref{Ltworho} shows that the image of $\Delta$ under the map $\wedge^2(\rH_1(U) \to \rH_1(X))$ is a valid choice for $\rho$, and we henceforth
let $\rho$ denote this image.

\subsection{The second graded quotient in the lower central series} \label{Ssecondquot}

By \eqref{Edelta}, $\Delta= \sum_{i=1}^g \bar{a}_i \wedge \bar{z}_i$.
Our goal is to determine $\Delta$ in terms of a basis of $\rH_1(U) \wedge \rH_1(U)$ for which 
we know the action of the absolute Galois group.
In order to do this, we investigate the element $T :=  \prod_{i=1}^g [a_i,z_i]$ in $\pi_1(U)$.
Note that $T = (c_0 \circ c_1 \circ \cdots c_{n-1})^{-1}$.

The next lemma shows that $\Delta$ does not depend on the representation
as a product of commutators.

\begin{lemma} \label{Lleeway}
Suppose $r_1, \ldots, r_N, s_1 \ldots, s_N$ are loops in $U$, 
with images $\bar{r}_i, \bar{s}_i$ in $\rH_1(U)$.
If
\[T \ {\rm is \ homotopic \ to \ } [r_1, s_1] \circ \cdots \circ [r_N, s_N],\] 
then $\sum_{i=1}^g \bar{a}_i \wedge \bar{z}_i = \sum_{i=1}^N \bar{r}_i \wedge \bar{s}_i$ in 
$\rH_1(U) \wedge \rH_1(U)$.
\end{lemma}

\begin{proof}
By hypothesis, in $\pi_1(U)$, 
\begin{equation}\label{prodrs=prodab} 
[a_1, z_1] \circ \cdots \circ [a_g, z_g] = [r_1, s_1] \circ \cdots \circ [r_N, s_N].\end{equation} 
Note that both sides of the equation are elements of $[\pi_1(U)]_2$. Therefore \eqref{prodrs=prodab} holds in $[\pi_1(U)]_2/[\pi_1(U)]_3$. 
Under the inverse of the isomorphism $C$, \eqref{prodrs=prodab} becomes 
$\sum_{i=1}^g \bar{a}_i \wedge \bar{z}_i = \sum_{i=1}^N \bar{r}_i \wedge \bar{s}_i$ in 
$\rH_1(U) \wedge \rH_1(U)$.
\end{proof}

The next lemma is key for simplifying later calculations.  

\begin{lemma} \label{Lsimplify}
Suppose $\alpha, \beta, \gamma \in \pi_1(U)$. 
\begin{enumerate}
\item If $\alpha \gamma \in [\pi_1(U)]_2$, then $\gamma \alpha \in [\pi_1(U)]_2$, and
$\alpha \gamma$ and $\gamma \alpha$ have the 
same image in $[\pi_1(U)]_2/[\pi_1(U)]_3$.

\item If $\gamma^{-1} \alpha \gamma \beta \in [\pi_1(U)]_2$, then $\alpha \beta \in [\pi_1(U)]_2$, and
the difference between the images of $\gamma^{-1} \alpha \gamma \beta$ and $\alpha \beta$ in 
$[\pi_1(U)]_2/[\pi_1(U)]_3$
is $\gamma \wedge (-\alpha)$.

\end{enumerate}
\end{lemma}

\begin{proof}

\begin{enumerate}
\item 
Note that $\gamma \alpha = \alpha^{-1} (\alpha \gamma) \alpha$.  
If $\alpha \gamma \in [\pi_1(U)]_2$, then $\gamma \alpha \in [\pi_1(U)]_2$ because the commutator subgroup is normal.
Also $\alpha \gamma$ and $\gamma \alpha$ have the 
same image in $[\pi_1(U)]_2/[\pi_1(U)]_3$ because conjugation acts trivially on $[\pi_1(U)]_2/[\pi_1(U)]_3$.

\item 
In $\pi_1(U)$, 
\[ [\gamma^{-1} \alpha \gamma, \gamma] = ( \gamma^{-1} \alpha \gamma) \gamma ( \gamma^{-1} \alpha^{-1} \gamma) \gamma^{-1}=\gamma^{-1} \alpha \gamma \alpha^{-1}=[\gamma^{-1}, \alpha].\]
So $[\gamma^{-1} \alpha \gamma, \gamma]\alpha \beta = \gamma^{-1} \alpha \gamma \beta$.
In particular, if $\gamma^{-1} \alpha \gamma \beta$ is in $[\pi_1(U)]_2$, then so is $\alpha \beta$.

Since $U$ is affine, $[\pi_1(U)]_2/[\pi_1(U)]_3 \cong \rH_1(U) \wedge \rH_1(U)$.
In $[\pi_1(U)]_2/[\pi_1(U)]_3$, the difference between the images of 
$\gamma^{-1} \alpha \gamma \beta$ and $\alpha \beta$ 
is the image of $[\gamma^{-1} \alpha \gamma, \gamma]$.
Since $[\gamma^{-1} \alpha \gamma, \gamma]=[\gamma^{-1}, \alpha]$, this image
is $-\gamma \wedge \alpha$, which equals $\gamma \wedge (-\alpha)$.
\end{enumerate}
\end{proof}

\subsection{Elements of the fundamental groupoid} \label{Sgroupoid}

Recall that $U:=U(\mathbb{C})=\{(x,y) \mid x^n + y^n =1\}$ and 
$Y$ is the set of $2n$ points such that $xy=0$. 
We compute in the fundamental groupoid $\pi_1(U, Y)$ of $U$ with respect to the base points in $Y$. 
Let $\beta$ be the path in $U$, which begins at the base point $b_0=b=(0,1)$ and ends at $d_0=(1,0)$, 
given by 
\begin{equation} \label{Ebeta}
\beta = (\sqrt[n]{t}, \sqrt[n]{1-t}) \ {\rm for} \ t \in [0,1].
\end{equation}

Throughout this section, let $0 \leq i \leq n-1$ and $0 \leq j \leq n-1$. 
It is sometimes convenient to think of $i$ and $j$ as elements of $\ZZ/n\ZZ$.
Let $b_j = (0, \zeta^j)$ and $d_i = (\zeta^i, 0)$.
Consider the automorphisms $\epzpi, \epopi \in {\rm Aut}(U)$ defined by 
$\epzpi(x,y) = (\zeta x, y)$ and $\epopi(x,y)=(x, \zeta y)$.  
Consider the path in $U$, which begins at $b_j$ and ends at $d_i$, given by
\begin{equation} \label{Ebetaij}
e_{i,j} = \epzpi^i \epopi^j \beta.
\end{equation}


Consider the loop $E_{i,j}$ in $U$, formed by the composition of four paths, 
where path composition is written from left to right:
\[E_{i,j} = e_{0,0} \circ  (e_{0,j})^{-1} \circ e_{i,j} \circ (e_{i,0})^{-1}.\]
Then $E_{i,j}$ proceeds through the following points: 
\[b_0 \mapsto d_0 \mapsto b_j \mapsto d_i \mapsto b_0.\]

If $ij=0$, then $E_{i,j}$ is trivial in the fundamental groupoid.  
The converse is also true; see Lemma \ref{LfactE} below.

\section{A formula for the classifying element} \label{S3}

The main goal of this section is to find a formula for $(c_0, c_1, \ldots, c_{n-1})^{-1}$
in terms of the elements $E_{i,j}$ in the fundamental groupoid $\pi_1(U, Y)$.
The formula is stated in Section~\ref{Sliftstar} and proved to be correct in Proposition~\ref{PS=C}.
The reason for finding this formula is that 
\[T=\prod_{i=1}^g [a_i, z_i]=(c_0 \circ c_1 \circ \cdots c_{n-1})^{-1}\] 
is in the class of the boundary of a disk in the Fermat curve 
that contains $Z_0$, the set of $n$ points where $z=0$.
At the end of the section, in Proposition~\ref{Pwrapin}, we analyze the ordering of the loops $E_{i,j}$ in $T$ combinatorially.
In Section~\ref{Smain}, we will use the material in this section and Section~\ref{Ssecondquot} 
to find an explicit formula for the 
element $\Delta \in \rH_1(U) \wedge \rH_1(U)$ whose image in $\rH_1(X) \wedge \rH_1(X)$ is $\rho$,
in terms of a basis on which we understand the action of the absolute Galois group.

\subsection{Sheets of a cover}

Let $V=\mathbb{A}^1(\mathbb{C})$.
Consider the map $\wp: U \to V$ given by $(x,y) \mapsto x$.
Let $\zeta=e^{2 \pi i/n}$.
Then $\wp$ is a $\mu_n$-Galois cover, where the generator $\zeta$ of $\mu_n$ acts via 
$\zeta(x,y)=(x,\zeta y)$.

The cover $\wp$ is ramified at 
$\{(x,y)=(\zeta^i,0) \mid i = 0, \ldots, n-1\}$ and 
branched at $\{x=\zeta^i \mid i = 0, \ldots, n-1\}$.
The pre-images of $x=0$ in $U$ are the points $b_j=(0,\zeta^j)$ for $0 \leq j \leq n-1$. 

The equation for the curve is $y^n=f(x)$ where $f(x) = -\prod_{i=0}^{n-1} (x - \zeta^i)^1$ is separable.
This implies that the inertia type of $\wp$ is the $n$-tuple $(1, 1, \ldots, 1)$. 
This means that the canonical generator of inertia at each ramification point is $\zeta$.
In other words, the chosen generator of the Galois group of $\wp$ acts on a uniformizer at 
each ramification point by the same root of unity $\zeta$.

Let $w_i$ be the path in $V$ given by $x=\zeta^i \sqrt[n]{t}$ for $t \in [0,1]$.
Let $V^\circ = V - \{w_i \mid 0 \leq i \leq n-1\}$.
Let $U^\circ = \wp^{-1} (V^\circ)$.
The restriction of $\wp$ to $U^\circ$ is unramified.
This is because the monodromy around each root of unity is multiplication by $\zeta$, 
so a loop going around all $n$ of the roots of unity is multiplication by $\zeta^n$, which is trivial. 
Therefore, the monodromy action of $\pi_1(V^\circ)$ is trivial on $U^\circ$, 
proving that the restriction of $\wp$ to $U^\circ$ is unramified.
Thus $U^\circ$ is a disjoint union of $n$ sheets.

In order to label these sheets, we define the following notation, for $0 \leq i \leq n-1$.
Let $r_i$ be the ray $x=\zeta^i \sqrt[n]{t}$ for $t \in [0,\infty)$.
We define the angular segment, or sector, $R_i$ to be the intersection of a small neighborhood of $x=0$ in $V^\circ$ 
with the segment of $V^\circ$ bounded by $r_{i-1 \bmod n}$ and $r_{i}$. 
In other words, for some small $\epsilon_\circ \in {\mathbb R}^{>0}$, 
\[R_i = \{x =re^{I \theta} \in V^{\circ} \mid (i-1) \frac{2 \pi}{n} < \theta < i \frac{2 \pi}{n}, 
\ 0 < r < \epsilon_\circ\}.\]
For $0 \leq i,j \leq n-1$, 
we denote by $\tilde{R}_{i,j}$ the intersection of a small neighborhood of $b_j=(0,\zeta^j)$ with $U^{\circ} \cap \wp^{-1}(R_i)$.
We label by $U_k$ the sheet containing $\tilde{R}_{k,0}$, for $0 \leq k \leq n-1$.

Thus our small neighborhood of $b_j=(0, \zeta^j)$ intersected with $U^\circ$ is the disjoint union 
of the $n$ neighborhoods $\tilde{R}_{i,j}$ for $0 \leq i \leq n-1$.
Since we may choose to base fundamental groups at a point $b_j$ or at a simply connected neighborhood of $b_j$, we can think of $b_j$ as a base point divided into $n$ pieces, one piece on 
each sheet $U_k$, or as $n$ different tangential basepoints.

In Lemma~\ref{Lwhichsheet}, we determine the value of $k$ such that the sheet $U_k$ contains 
$\tilde{R}_{i,j}$ when $j \not = 0$.

\subsection{The cusps with $z=0$}

Recall that $Z_0$ is the set of $n$ points of $X$ where $z=0$.
Let $\epsilon = e^{\pi I/n}$ be a primitive $n$th root of $-1$.
The points of $Z_0$ have projective coordinates 
$z_k = [\zeta^{k} \colon \epsilon \colon 0]$ for $0 \leq k \leq n-1$.
The unramified cover $U^\circ \to V^\circ$ extends to an unramified cover 
$U^\circ \cup Z_0 \to V^{\circ} \cup \{\infty\}$.
The next result shows that the boundary of the sheet $U_k$ contains exactly one point of $Z_0$. 
 
\begin{lemma} \label{Lpointinfinity}
The point $z_k$ is contained in the boundary of the sheet $U_k$.
\end{lemma}

\begin{proof}
Consider the ray $Q_k$ in $U^\circ$ given by $(x,y) =(\epsilon^{-1} \zeta^k \sqrt[n]{t}, \sqrt[n]{1+ t})$ for $t \in (0, \infty)$.
As $t \to 0$, it approaches $b=(0,1)$ and it is contained in the sheet $U_k$.  
As $t \to \infty$, the value of $x/y$ on $Q_k$ approaches 
\[{\rm lim}_{t \to \infty} \epsilon^{-1} \zeta^{k} \sqrt[n]{t}/\sqrt[n]{1+t} = \epsilon^{-1} \zeta^{k}.\] 
The point $z_{k}$ is at the end of the ray $Q_k$ and is thus contained in the sheet 
$U_k$.  
\end{proof}

\begin{figure}[h!]
  \caption{}
  \label{figure1}
  \vspace{1pt}
  \fbox{ \includegraphics[width=0.18\textwidth]{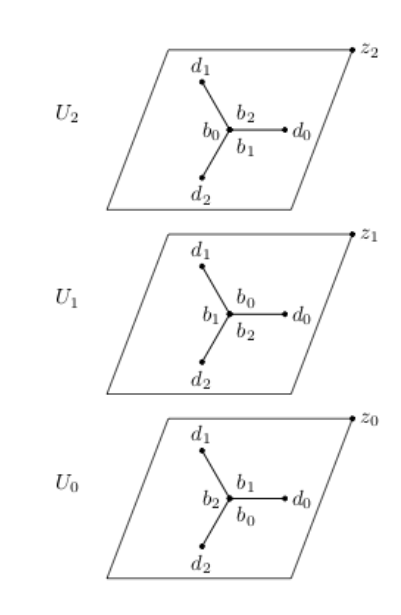}}
\end{figure}

\subsection{Loops in the cover} \label{Sinertiatype}

The information from the inertia type indicates how to glue the sheets $U_k$ together along the 
paths $e_{i,j}$ to obtain a ramified cover of Riemann surfaces.
Recall that the canonical generator of inertia at each ramification point is $\zeta$.

Consider a loop $L_i$ in $V$, with starting point $x=0$, which makes a 
counterclockwise circle around the path $w_i$, starting in the region $R_{i}$ and ending in $R_{i +1}$.
Note that $L_i$ depends only on the value of $i$ modulo $n$.
Define $L_{i,j}$ to be the lift of $L_i$ to a path in $U^\circ$ which starts at the point $b_j=(0, \zeta^j)$.
By the definition of the inertia type, $L_{i,j}$ ends at the point $b_{j+1}=(0, \zeta^{j+1})$. 
Note that $L_{i,j}$ is homotopic to the composition of paths 
\[L_{i,j} = e_{i,j} \circ (e_{i,j+1})^{-1} \colon
(0, \zeta^j) \mapsto (\zeta^i, 0) \mapsto (0, \zeta^{j+1}).\]

The next result determines which of the lifts $\tilde{R}_{i,j}$ of the sector $R_i$ are on each sheet $U_k$.

\begin{lemma} \label{Lwhichsheet}
With notation as above, $\tilde{R}_{i,j} \subset U_k$ if and only if $i-j = k$.
\end{lemma}

\begin{proof}
By definition, $\tilde{R}_{k, 0} \subset U_k$.  
The path $L_{i,j}$ is contained in a unique sheet.
Since $L_{i,j}$ starts in $\tilde{R}_{i, j}$ and ends in $\tilde{R}_{i+1, j+1}$, 
then these neighborhoods are contained in the same sheet.
\end{proof}

Note that $U^\circ = U - \{e_{i,j} \mid 0 \leq i,j \leq n-1\}$.
In order to reconstruct the ramified cover of Riemann surfaces $U \to V$, 
we need to glue the sheets $\{U_k\}_{0 \leq k \leq n-1}$ together along the missing segments;
specifically, we glue $R_{i,j}$ and $R_{i+1, j}$ together along the edge $e_{i,j}$ for $0 \leq i,j \leq n-1$.

\subsection{Lifting of a star shape} \label{Sliftstar}

Recall that path composition is written from left to right.
For $0 \leq \ell \leq n-1$, define a loop in $V$ by:
\begin{equation} \label{EdefS}
S_{\ell}=L_{n-\ell} \circ L_{n-\ell + 1} \circ \cdots \circ L_{n-\ell + (n -1)}.
\end{equation} 
Each loop $S_\ell$ traces in a counterclockwise direction 
along the outside of the slits $\{w_i\}$ and forms an $n$-pointed star shape.  
Each is homotopic to a large circle in $V^\circ$ traced in a counterclockwise direction.

\begin{definition} \label{Dstarlift}
For $0 \leq \ell \leq n-1$, let $\tilde{S}_\ell$ denote the unique lift under $\wp$ of $S_{\ell}$ 
to a loop in $U$ with starting point $b_0=(0,1)$.
Let $\tilde{S}:=\tilde{S}_0 \circ \tilde{S}_{1} \cdots \circ \tilde{S}_{n-1}$.
\end{definition}

By the proof of Lemma~\ref{Lwhichsheet}, $\tilde{S}_{\ell}$ is contained in $U_{n-\ell}$.
Later, we will see that $\tilde{S} \in [\pi_1(U)]_2$; see Remark \ref{RSandgamma}.  

\begin{lemma} \label{Lorderloop}
For $0 \leq \ell \leq n-1$, the loop $\tilde{S}_\ell$ is homotopic to
\[\tilde{S}_\ell= e_{-\ell, 0} \circ (e_{-\ell, 1})^{-1} \circ e_{-\ell+1, 1} \circ (e_{-\ell + 1, 2})^{-1} \circ \cdots \circ 
e_{-\ell + n-1, n-1} \circ (e_{-\ell + n-1, 0})^{-1},\]
or, equivalently,
$L_{n-\ell,0} \circ L_{n-\ell+1, 1} \circ L_{n-\ell+2,2} \circ \cdots \circ L_{n-\ell+(n-1), n-1}$.
\end{lemma}

\begin{proof}
The loop $\tilde{S}_\ell$ is homotopic to the composition of $2n$ of the edges $e_{i,j}$ and $(e_{i,j})^{-1}$.
Because of the inertia type of $\wp$, this composition involves 
loops of the form $L_{i,j}=e_{i,j} \circ (e_{i,j+1})^{-1}$.
The condition that $\tilde{S}_\ell$ is contained in $U_{n-\ell}$ implies that its
initial edge is the path $e_{-\ell, 0}$ from $(0,1)$ to $(\zeta^{-\ell},0)$.
Thus the initial loop is $L_{n-\ell,0}=e_{-\ell, 0} \circ (e_{-\ell, 1})^{-1}$.
Consider the loop $L_{i',j'}$ coming after the loop $L_{i,j}$.  
Then $i'=i+1$ because $\tilde{S}_\ell$ circles counterclockwise around the 
point $(\zeta^i, 0)$.
Also $j'=j+1$ because the starting point $(0, \zeta^{j'})$ of $L_{i',j'}$ is the same 
as the ending point $(0, \zeta^{j+1})$ of $L_{i,j}$.
\end{proof}

For example, $\tilde{S}_0$ passes around the points in this order:
\[(0,1) \to (1,0) \to (0, \zeta) \to (\zeta,0) \to (0, \zeta^2) \to \cdots \to (\zeta^{n-1}, 0) \to (0,1),\]
and 
\[\tilde{S}_0 = e_{0,0} \circ (e_{0,1})^{-1} \circ e_{1,1} \circ (e_{1,2})^{-1} \circ \cdots \circ 
e_{n-1, n-1} \circ (e_{n-1, 0})^{-1}.\]




\subsection{Comparison with loops around cusps with $z=0$} \label{Sboundaryloop}

We prove that 
$\tilde{S}$ is path homotopic to the boundary of a disk containing the $n$ cusps of $Z_0$ and that
$\tilde{S}_{\ell}$ can be taken to be the loop $(c_{n-\ell})^{-1}$ in a standard presentation of $\pi_1(U)$ 
(a presentation of the form~\eqref{pi1Upresentation}).
For $0 \leq j \leq n-1$, let $z_j=[\zeta^j \colon \epsilon \colon 0]$.  

\begin{proposition} \label{Pwhichpoint}
The loop $\tilde{S}_\ell$ is homotopic to the clockwise loop bounding a disk containing $z_{n - \ell}$.
\end{proposition}

\begin{proof}
Note that $V^\circ$ is homeomorphic to ${\mathbb A}^1 - \{0\}$ and $S_{\ell}$
is homotopic to a counterclockwise loop around $0$.  
Thus $S_{\ell}$ is homotopic to a clockwise loop around $\infty$.
By definition, the lift $\tilde{S}_{\ell}$ of $S_{\ell}$ is a loop in $U_{n-\ell}$.
The restriction of $\wp$ to $U_{n-\ell}$ yields a homeomorphism $U_{n-\ell} \to V^\circ$.
Thus $\tilde{S}_{\ell}$ is a clockwise loop around the point missing from $U_{n-\ell}$.
By Lemma \ref{Lpointinfinity}, this point is $z_{n-\ell}$. 
\end{proof}
 
\begin{proposition} \label{PS=C}
The loop $\tilde{S}$ is homotopic to the 
boundary of a disk in the Fermat curve which contains the $n$ points where $z=0$;
it follows that we may take a presentation of the form \eqref{pi1Upresentation} where $\tilde{S}_\ell =(c_{n-\ell})^{-1}$ and  $\tilde{S}=(c_1 \circ \cdots \circ c_{n-1} \circ c_0)^{-1}$.
\end{proposition}

In the following proof, it is convenient to use a small ball around $b_0$ as our basepoint $b_0$ (which we may do because balls are simply connected). The intersection of this ball with $\tilde{R}_{i,0}$ will then be called {\em the fractional point of $b_0$ 
in the region $\tilde{R}_{i,0}$.}

\begin{proof}
The loop $\tilde{S}_0$ in $U_0$ starts and ends at the fractional point $b_0$ 
in the region $\tilde{R}_{0,0}$.
With a small homotopy adjustment, the end of $\tilde{S}_0$ can cross the path $e_{n-1,0}$ rather than return to the point $b_0$. 
Since the sheet $U_0$ is glued to the sheet $U_{n-1}$ along the edge $e_{n-1, 0}$, 
this yields an ending point in the region $\tilde{R}_{n-1,0}$ 
near the fractional point $b_0$ in the sheet $U_{n-1}$.

The loop $\tilde{S}_{1}$ in $U_{n-1}$ 
starts and ends at the fractional point $b_0$ 
in the region $\tilde{R}_{n-1,0}$.
With a small homotopy adjustment, the end of $\tilde{S}_{1}$
can cross the path $e_{n-2,0}$ rather than return to the point $b_0$. 
Since the sheet $U_{n-1}$ is glued to the sheet $U_{n-2}$ along the edge $e_{n-2, 0}$, 
this yields an ending point in the region $\tilde{R}_{n-2,0}$ 
near the fractional point $b_0$ in the sheet $U_{n-2}$.

We continue in this way through all the loops $\tilde{S}_0 \circ \tilde{S}_{1} \circ \cdots \circ \tilde{S}_{n-1}$.  Finally, with a small homotopy adjustment, the end of 
$\tilde{S}_{n-1}$ crosses the path $e_{0,0}$ and returns to the region $\tilde{R}_{0,0}$ in $U_0$.  Thus 
$\tilde{S}$ is path homotopic to a loop in the Fermat curve $X$ composed of $n$ paths each contained on a single $U_k$ of the form shown in Figure~\ref{figure2} for $n=3$.

This path divides $X$ into an external and internal piece, where the internal piece
contains the ramification points $\{d_i \mid 0 \leq i \leq n-1\}$, and the external piece contains $Z_0$ and is homeomorphic to a disk.
(To see that the external piece is homeomorphic to a disk, note the following. The external piece is the union of $n$ pieces, each homotopic to a wedge, which are 
glued together along the edges of the wedges.
The picture for $n=3$ is illustrated in Figure~\ref{figure2}.)

Applying Proposition \ref{Pwhichpoint}, it follows that we can take a presentation of the form
\eqref{pi1Upresentation} where $\tilde{S}_\ell =(c_{n-\ell})^{-1}$.
Thus $\tilde{S}=(c_0)^{-1} \circ (c_{n-1})^{-1} \circ  \cdots \circ (c_1)^{-1}=(c_1 \circ \cdots \circ c_{n-1} \circ c_0)^{-1}$.
\end{proof}

\begin{figure}[h!]
  \caption{}
  \label{figure2}
  \vspace{1pt}
 \fbox{ \includegraphics[width=0.21\textwidth]{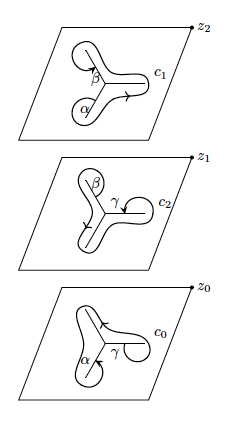}}
\end{figure}

The difference between $(c_0 \circ c_1 \circ \cdots c_{n-1})^{-1}$ 
and $(c_1 \circ \cdots \circ c_{n-1} \circ c_{0})^{-1}$ is not significant by
Lemma~\ref{Lsimplify}(1).

\begin{lemma} \label{LetoE}
Let $S^*$ (resp.\ $S^*_\ell$) be the element of $\pi_1(U)$ obtained by substituting $E_{i,j}$ for the path $e_{i,j}$ in $\tilde{S}$ (resp.\ $\tilde{S}_\ell$).
Then $S^*$ and $\tilde{S}$ (resp.\ $S^*_\ell$ and $\tilde{S}_\ell$) 
are homotopic in $\pi_1(U)$.
\end{lemma}

\begin{proof}
We fix the path $e_{0,0} \circ e_{0,j}^{-1}$ from $b_0$ to $b_j$, the initial point of $e_{i,j}$.
We fix the path $e_{i,0}$ from $b_0$ to $d_i$, the final point of $e_{i,j}$.
In the composition of paths, after completing each path $e_{i,j}$, 
we return to the base point $b_0$ and then return to the initial point of the next path.
This does not change the homotopy class.
The statement can also be proven using algebraic cancelation.
\end{proof}

\subsection{Combinatorial analysis}

By Proposition~\ref{PS=C}, the loop $\tilde{S}$ is homotopic to the 
boundary of a disk in the Fermat curve which contains the $n$ points where $z=0$.
By Lemma~\ref{LetoE}, the loop $S^*$ also has this property.
So the goal is to find the image of $S^*$ in $[\pi_1(U)]_2/[\pi_1(U)]_3$ 
in terms of the elements $E_{i,j}$.
By Lemma~\ref{Lsimplify}(2), this is possible if we have a complete understanding of the 
edges in between $E_{i,j}^{-1}$ and $E_{i,j}$ in $S^*$.  We describe this combinatorially in this section.

\begin{lemma} \label{LfactE2}
The loop $\tilde{S}$ is the composition of $2n^2$ paths, 
with each path $e_{i,j}$ and each path $(e_{i,j})^{-1}$ occuring exactly once.
\end{lemma}

\begin{proof}
Immediate from Lemma \ref{Lorderloop}.
\end{proof}

We begin a combinatorial analysis of the ordering of the elements 
$E_{i,j}$ and $E_{i,j}^{-1}$ in $S^*$, viewed as a 
cycle, rather than a word. 
For $1 \leq j \leq n-1$ and $0 \leq a \leq n-2$, 
let $\overline{j+a}$ be the unique value in $\{1, \ldots, n-1\}$ congruent to $j+a$ modulo $n-1$.

\begin{proposition} \label{Pwrapin}
The ordering of the elements $E_{i,j}$ and $E_{i,j}^{-1}$ in the cycle $S^*$ 
satisfies the following:
\begin{enumerate}
\item The edges between $E_{1,1}^{-1}$ and $E_{1,1}$ are: 
\[f_1:=E_{2,1} \circ (E_{2,2})^{-1} \circ \cdots \circ E_{n-1, n-2} \circ (E_{n-1,n-1})^{-1}.\]
\item The edges between $E_{1,j}^{-1}$ and $E_{1,j}$ are: 
\[f_j:=E_{2, \overline{j}} \circ (E_{2, \overline{j+1}})^{-1} \circ \cdots \circ E_{n-1, \overline{j + n-3}} \circ (E_{n-1, \overline{j + n-2}})^{-1}.\]
\item For $1 \leq i,j \leq n-1$, the edges $E_{i',j'}$ between $E_{i,j}^{-1}$ and $E_{i,j}$ 
with $i' \geq i$ are:\\
$E_{i', j'}$ with $i < i' \leq n-1$ and $i'-j' \equiv i-j+1 \bmod n-1$; and\\ 
$(E_{i',j'})^{-1}$ with $i < i' \leq n-1$ and $i'-j' \equiv i-j \bmod n-1$.
\end{enumerate}
\end{proposition}

\begin{proof}
Item (1) is a special case of (2), which is a special case of (3), which follows from 
Lemma \ref{Lorderloop}.
\end{proof}

\section{The homology of the Fermat curve}\label{Sstructure_of_homology} \label{S4}

The homology of the Fermat curve has been studied from many perspectives, see e.g., 
\cite[appendix]{grossrohrappen} and \cite[Section 4]{lim}.
By \cite[Theorem 1, appendix]{grossrohrappen}, $\rH_1(X)$ is a cyclic $\Lambda_1$-module, 
where $\Lambda_1=\ZZ[\mu_n \times \mu_n]$, 
and the annihilator of $\rH_1(X)$ in $\Lambda_1$ can be found in 
\cite[page 210]{grossrohrappen} and \cite[Proposition 4.1]{lim}.
The facts about the structure of $\rH_1(U)$ and $\rH_1(X)$ in this section 
will be familiar to the experts.

In Sections~\ref{S4} and \ref{Smain}, we consider homology with coefficients in $\ZZ$;
however, we follow an approach which is compatible with the results in \cite{Anderson}, 
\cite{WINF:birs}, and \cite{Baction} about the \'etale homology with coefficients in $\ZZ/n\ZZ$ and 
the action of the absolute Galois group upon it; see Section \ref{S5}.

Consider the relative homology $\rH_1(U,Y)$, viewed as a groupoid with one object.
To study $\rH_1(U)$, we use the homomorphism of groupoids from
$\pi_1(U, Y)$ to $\rH_1(U,Y)$, which sends composition to addition; we denote this homomorphism by 
$- \mapsto [-]$. 

\subsection{A basis for the homology of $X$} \label{subsectionwithEijdef}

Let $\Lambda_1 = \ZZ[\mu_n \times \mu_n]$ and let $\epz$ and $\epo$ 
denote the generators of $\mu_n \times \mu_n$.
Let $A_1 = \langle (\epz-1)(\epo-1) \rangle \subset \Lambda_1$ denote the augmentation ideal.

Let $[e_{i,j}]$ (resp.\ $[E_{i,j}]$) denote the class of $e_{i,j}$ (resp.\ $E_{i,j})$ 
in the relative homology $\rH_1(U, Y)$.
Let $\beta$ denote the class of $[e_{0,0}]$ and 
recall that $[e_{i,j}] = \epz^i \epo^j \beta$.
Also
\begin{equation} \label{EEij}
[E_{i,j}] = [e_{0,0}] - [e_{i,0}] - [e_{0,j}] + [e_{i,j}].
\end{equation}

Using modular symbols, Ejder proves in \cite[Theorem 1.2]{E:integral} that a basis for $\rH_1(X)$ 
is given by 
\begin{equation}
\{\epz^i \epo^j(1-\epz)(1-\epo)\beta \mid 1 \leq i \leq n-2, \ 0 \leq j\leq n-2\}.
\end{equation}

In our notation, that means that a basis for $\rH_1(X)$ is given by
\begin{equation} \label{Ehomproj} 
\{\epz^i \epo^j[E_{1,1}] \mid 1 \leq i \leq n-2, \ 0 \leq j\leq n-2\}.
\end{equation}

\subsection{Facts about the homology of the affine curve}\label{subsectionwithH1U}

Next we find a basis for $\rH_1(U)$.

\begin{lemma} \label{LfactE}
The elements $[E_{i,j}]$ from \eqref{EEij} are in $\rH_1(U)$ and
the set $\{[E_{i,j}] \mid 1 \leq i,j \leq n-1\}$ is a basis for $\rH_1(U)$ as a $\Z$-module.
\end{lemma}

\begin{proof}
The first claim is true because $[E_{i,j}]$ is the image of a path in the fundamental groupoid starting and ending at the same point. 

We first show that $\{[E_{i,j}] \mid 1 \leq i,j \leq n-1\}$ is a basis for $\rH_1(U; \ZZ/n\ZZ)$.
There is an isomorphism $\rH_1(U,Y; \ZZ/n\ZZ) \cong \Lambda_1 \otimes (\ZZ/n\ZZ)$, 
taking $\beta \mapsto 1$, \cite[Theorem 6]{Anderson}. 
Thus $\{[e_{i,j}] \mid 0 \leq i,j \leq n-1\}$ is a basis for $\rH_1(U, Y; \ZZ/n\ZZ)$.

Consider the augmentation ideal $A_1 \otimes \ZZ/n\ZZ \subset \Lambda_1 \otimes \ZZ/n\ZZ$.
If $\alpha \in \Lambda_1 \otimes \ZZ/n\ZZ$, write $\alpha = \sum_{i,j} a_{i,j} \epz^i \epo^j$.  
One can check that  $\alpha \in A_1 \otimes \ZZ/n\ZZ$ if and only if the rows and columns of the matrix 
$[a_{i,j}]$ sum to $0$ modulo $n$.
By \cite[Proposition 6.2]{WINF:birs}, $\rH_1(U;\ZZ/n\ZZ) \cong A_1 \otimes \ZZ/n\ZZ$.

The element $[e_{i,j}]$ appears in $[E_{i',j'}]$ 
if and only if $i'=i$ and $j'=j$.
It follows that $\{[E_{i,j}] \mid 1 \leq i,j \leq n-1\}$ is a set of linearly independent elements in $\rH_1(U, Y; \ZZ/n\ZZ)$, 
and thus also in $\rH_1(U; \ZZ/n\ZZ)$.  
Their span contains $n^{((n-1)^2)}$ elements of $\rH_1(U;\ZZ/n\ZZ)$.
Since $\rH_1(U;\ZZ/n\ZZ)$ has rank $(n-1)^2$, this span is the entirety of $\rH_1(U; \ZZ/n\ZZ)$.
This completes the proof that $\{[E_{i,j}] \mid 1 \leq i,j \leq n-1\}$ is a basis for $\rH_1(U;\ZZ/n\ZZ)$.

It follows that $\{[E_{i,j}] \mid 1 \leq i,j \leq n-1\}$ is a set of linearly independent elements in 
$\rH_1(U)$.  Consider the span of the image of this set in $\rH_1(X)$.
This span contains $[E_{1,1}]$ and is a $\Lambda_1$-module, thus contains $\epz^i \epo^j[E_{1,1}]$.
By \eqref{Ehomproj}, the image of this set spans $\rH_1(X)$. 

To complete the proof, we need to show that $\{[E_{i,j}] \mid 1 \leq i,j \leq n-1\}$ spans the 
kernel of $\rH_1(U) \to \rH_1(X)$.  A basis for the kernel is 
$\{\bar{c_j} \mid 0 \leq j \leq n-1\}$.
By Proposition \ref{Pwhichpoint}, 
the loop $\tilde{S}_\ell$ is homotopic to the clockwise loop bounding a disk containing $z_{n - \ell}$.
Thus a basis for the kernel is the set of images of $\tilde{S}_\ell$ in $\rH_1(U)$.
By Lemma \ref{Lorderloop}, $\tilde{S}_\ell$ is homotopic to a loop with a formula written in terms of $e_{i,j}$.
By Lemma \ref{LetoE}, the same is true after replacing $e_{i,j}$ by $E_{i,j}$.  Also $E_{i,j} =0$ if $ij=0$.
Thus the set of images of $\tilde{S}_\ell$ in $\rH_1(U)$ is generated by $\{[E_{i,j}] \mid 1 \leq i, j \leq n-1\}$.
This completes the proof.
\end{proof}

By Lemma \ref{LfactE}, 
there is an injection $\rH_1(U) \wedge \rH_1(U) \to \Lambda_1 \wedge_{\Z} \Lambda_1$
and an isomorphism
\[\rH_1(U) \wedge \rH_1(U) \to A_1 \wedge_{\Z} A_1.\]

\begin{lemma} \label{Lbasiswedge}
Consider the index set
\begin{equation} \label{indexset}
I = \{(i_1,j_1, i_2,j_2) \mid 1 \leq i_1, i_2, j_1,j_2 \leq n-1, \ i_1 \leq i_2, \ {\rm and \ if} \ i_1=i_2 \ {\rm then} \ j_1 < j_2\}.
\end{equation}
Then $\rH_1(U) \wedge \rH_1(U)$ is a free $\ZZ$-module with basis 
$\{[E_{i_1,j_1}] \wedge [E_{i_2,j_2}] \mid (i_1,j_1,i_2,j_2) \in I\}$.
\end{lemma}

\begin{proof}
By Lemma \ref{LfactE}, $\rH_1(U)$ is a free $\ZZ$-module of rank $m:=(n-1)^2$ with basis 
$\{[E_{i_1,j_1}] \mid 1 \leq i_1,j_1 \leq n-1\}$.
Then $\rH_1(U) \wedge \rH_1(U)$ is a free $\ZZ$-module of rank $\binom{m}{2}$.
Because $z \wedge w = - w \wedge z$ and $z \wedge z = 0$, 
a basis is given by the set of simple wedges $[E_{i_1, j_1}] \wedge [E_{i_2,j_2}]$
with $i_1 \leq i_2$ and $(i_1,j_1) \not = (i_2, j_2)$, which is indexed by $I$.
\end{proof}

\subsection{Facts about the homology of the projective curve}

We characterize $\rH_1(X) \wedge \rH_1(X)$ as a quotient of $\rH_1(U) \wedge \rH_1(U)$ both for theoretical reasons
and for the computational applications in Sections~\ref{Sverify}-\ref{Sapplication}.

\begin{lemma} \label{Lwedge}
Let $S$ be the kernel of $\rH_1(U) \to \rH_1(X)$.
Then the kernel of $\rH_1(U) \wedge \rH_1(U) \to \rH_1(X) \wedge \rH_1(X)$
equals $S \wedge \rH_1(U)$.
\end{lemma}

\begin{proof}
Since $\rH_1(X)$ is a free module, the quotient map $\rH_1(U) \to \rH_1(X)$ splits, giving a direct sum decomposition $\rH_1(U) \cong \rH_1(X) \oplus S$, where $S$, $\rH_1(X)$ and $\rH_1(U)$ are all free modules. The wedge of the direct sum decomposes as $\rH_1(U) \wedge \rH_1(U)  \cong (\rH_1(X) \wedge \rH_1(X)) \oplus 
(\rH_1(X) \wedge S) \oplus (S \wedge S)$, showing the claim.
\end{proof}

We need an explicit description of $\rH_1(X)$ as a quotient of $\rH_1(U)$
for the computational applications in Sections~\ref{Sverify}-\ref{Sapplication}.
Define $\gamma_i \in \Lambda_1$ by the formula
\begin{equation} \label{Egamma}
\gamma_i = \epz^{-i} (1 - \epo)(1 + \epz \epo + \cdots + \epz^{n-1} \epo^{n-1}).
\end{equation}

\begin{lemma} \label{Lrelationgamma} 
The set $\{\gamma_i \beta\mid 1 \leq i \leq n-1\}$ is a basis for $S={\rm Ker}(\rH_1(U) \to \rH_1(X))$.
\end{lemma}

\begin{proof}
By Proposition \ref{Pwhichpoint}, 
a basis for the kernel of $\rH_1(U) \to \rH_1(X)$ is the set of images of $\tilde{S}_\ell$ in $\rH_1(U)$.
Using Lemma \ref{Lorderloop}, one can check that $\gamma_\ell \beta$ is 
the image of $\tilde{S}_\ell$ under the map $\pi_1(U) \to \rH_1(U)$.
Thus $\{\gamma_i \beta\mid 1 \leq i \leq n-1\}$ is a basis 
for the kernel of $\rH_1(U) \to \rH_1(X)$.
\end{proof}

\begin{remark} \label{RSandgamma}
Since $\gamma_\ell \beta$ is 
the image of $\tilde{S}_\ell$, one can see that
the image of $\tilde{S}$ is 
$\sum_{\ell=0}^{n-1} \gamma_\ell \beta=0$.  Thus $\tilde{S} \in [\pi_1(U)]_2$.
 \end{remark}

\subsection{Properties of the classifying element $\rho$} \label{Sproperties}

The classifying element $\rho \in \rH_1(X) \wedge \rH_1(X)$ satisfies the following invariance property 
for the geometric action of automorphisms in ${\rm Aut}(X)$.  See Proposition \ref{Pabstractalg}
for an invariance property for the arithmetic action of automorphisms in the absolute Galois group $G_\QQ$.

\begin{proposition} \label{Pabstractinv}
Let $\rho$ be a generator for the image of $\rH_2(X) \to \rH_1(X) \wedge \rH_1(X)$.
If $\phi \in {\rm Aut}(X)$, then $\phi(\rho)=\rho$.
\end{proposition}

\begin{proof}
Every algebraic automorphism $\phi$ of $X$ is orientation-preserving and 
preserves the fundamental class in $\rH_2(X)$. It follows that $\phi$ preserves the fundamental class in $\rH_2(X)$.
\end{proof}

By \cite{Tzermias95}, or \cite{Leopoldt}, 
if $n \geq 4$, then $|{\rm Aut}(X)|=6n^2$.
We will apply Proposition~\ref{Pabstractinv} to:
\begin{enumerate}
\item the automorphisms $\phi_0([x:y:z])=[\zeta x: y: z]$ and $\phi_1([x:y:z])=[x: \zeta y: z]$
which act on $\rH_1(U,Y)$ via multiplication by $\epz$ and $\epo$ respectively.
So $\rho$ is invariant under the action of $\Lambda_1$.
\item the transposition $\tau([x:y:z])=[y:x:z]$; after using $\beta$ to fix an isomorphism between $\Lambda_1$ and $\rH_1(U,Y)$, then $\beta$ acts on $\rH_1(U,Y)$ by 
the ring automorphism of $\Lambda_1$ that switches $\epz$ and $\epo$.
So $\rho$ is symmetric.
\item the $3$-cycle $\omega([x:y:z]) = [z:-x:y]$; this automorphism does not stabilize $U$ and $Y$.
\end{enumerate}

\section{Main result} \label{Smain}

In this section, we complete the analysis of the structure of ${\rm gr}(\pi)$ as a graded Lie algebra, by 
finding a formula for the element $\Delta \in \rH_1(U) \wedge \rH_1(U)$ that maps to $\rho \in 
\rH_1(X) \wedge \rH_1(X)$.  Then we give some examples for $n=3,4,5$.

\subsection{Proof of the main result}
By Lemma \ref{Lbasiswedge}, with $I$ defined as in \eqref{indexset}, 
$\rH_1(U) \wedge \rH_1(U)$ is a free $\ZZ$-module with basis $\{[E_{i_1,j_1}] \wedge [E_{i_2,j_2}] \mid (i_1,j_1,i_2,j_2) \in I\}$. Thus there exist $\epsilon(i_1,j_1,i_2,j_2) \in \ZZ$, such that
$\Delta \in \rH_1(U) \wedge \rH_1(U)$ can be uniquely represented as the linear combination 
\[\Delta = \sum_{(i_1,j_1,i_2,j_2) \in I} \epsilon(i_1,j_1,i_2,j_2) [E_{i_1,j_1}] \wedge [E_{i_2,j_2}].\]

Theorem \ref{Tintro} follows immediately from the next result.

\begin{theorem} \label{Tmainresulthere}
In $\rH_1(U) \wedge \rH_1(U)$, 
the coefficient $\epsilon(i_1,j_1,i_2,j_2)$ of the basis element $[E_{i_1,j_1}] \wedge [E_{i_2, j_2}]$ in $\Delta$ is 
\[\epsilon(i_1,j_1,i_2,j_2) =\begin{cases} 
      1 & \text{if } j_2-j_1 \equiv i_2-i_1 \not \equiv 0 \bmod n-1,  \\
     -1 & \text{if } j_2-j_1 +1 \equiv i_2-i_1 \not \equiv 0 \bmod n-1, \\
      0 & \text{otherwise}.
   \end{cases}
\]
\end{theorem}

\begin{proof}
Recall that $T=\prod_{i=1}^g [a_i, z_i]=(c_0 \circ c_1 \circ \cdots c_{n-1})^{-1}$.
By Lemma \ref{Lleeway}, if
$r_1, \ldots, r_N, s_1 \ldots, s_N \in \pi_1(U)$ are such that 
$T= [r_1, s_1] \circ \cdots \circ [r_N, s_N]$,
then $\Delta = \sum_{i=1}^N \bar{r}_i \wedge \bar{s}_i$ in $\rH_1(U) \wedge \rH_1(U)$.
By Proposition \ref{PS=C},
$\tilde{S}$ is homotopic to $(c_1 \circ \cdots \circ c_{n-1} \circ c_0)^{-1}$.
The difference between $(c_0 \circ c_1 \circ \cdots c_{n-1})^{-1}$ 
and $(c_1 \circ \cdots \circ c_{n-1} \circ c_{0})^{-1}$ is not significant by
Lemma~\ref{Lsimplify}(1); thus $\tilde{S}$ and $T$ have the same image in 
$\rH_1(U) \wedge \rH_1(U)$.

By Lemma~\ref{LetoE}, $S^*$ is homotopic to $\tilde{S}$.
It thus suffices to express $S^*$ as a product of commutators.
By Lemma~\ref{Lsimplify}(2), the image of $S^*$ in $[\pi_1(U)]_2/[\pi_1(U)]_3$ 
depends only on the ordering of the edges in between $E_{i,j}^{-1}$ and $E_{i,j}$ in $S^*$. 
We may view $S^*$ as a cycle rather than a word, meaning that the last element precedes the first one.
By Proposition \ref{Pwrapin}, 
the ordering of the elements $E_{i,j}$ and $E_{i,j}^{-1}$ in $S^*$ is:
\[(E_{1,1}^{-1} \circ f_1 \circ E_{1,1}) \circ (E_{1,2}^{-1} \circ f_2 \circ E_{1,2}) \circ \cdots 
\circ (E_{1, n-1}^{-1} \circ f_{n-1} \circ E_{1,n-1}),\]
where $f_j$ is defined in Proposition \ref{Pwrapin}.

By Lemma \ref{Lsimplify}(2), $\Delta$ is the sum of
\begin{equation} \label{EforE1}
[E_{1,1}] \wedge (-f_1) + \cdots + [E_{1,n-1}] \wedge (-f_{n-1}),
\end{equation}
and the image of $\tilde{f} := f_1 \circ \cdots \circ f_{n-1}$ in $\rH_1(U) \wedge \rH_1(U)$.

Note that $E_{1,j}$ and $E_{1,j}^{-1}$ do not appear in $\tilde{f}$ for any $j$.
Thus the coefficient of $[E_{1,j}] \wedge [E_{i_2,j_2}]$ is zero unless $E_{i_2,j_2}$ or its inverse appears in $f_j$.
In particular, it is zero if $i_2 = 1$.
For $i_2 \not = 1$, by the definition of $f_j$, the coefficient of $[E_{1,j}] \wedge [E_{i_2,j_2}]$ is $+1$ if $j_2-i_2= j-1$ and is $-1$ if 
$j_2-i_2=j-2$.  This is equivalent to the coefficient being $+1$ if $j_2-j_1 \equiv i_2 - 1 \not \equiv 0 \bmod n-1$ 
and being $-1$ if $j_2 -j_1 +1 \equiv i_2 - i_1 \not \equiv 0 \bmod n-1$, which is the claimed statement for $i_1 = 1$.

Furthermore, the ordering of the edges in the cycle $\tilde{f}$ is the same as for 
the cycle $S^*$, 
except the edges $e_{1,j}$ and $e_{1,j}^{-1}$ do not appear.
Using Proposition \ref{Pwrapin}(3) and repeating the argument shows that the statement is 
true for $i=2$.  The result follows by induction. 
\end{proof}

\begin{remark} \label{Rcheck}
For $p=5$, we were able to independently verify using Magma that the image of $\Delta$ generates 
$\langle \rho \rangle$ in $\rH_1(X) \wedge \rH_1(X)$; see Section~\ref{Sverify}. 
This uses the invariance properties from Propositions \ref{Pabstractinv} and \ref{Pabstractalg} and
the explicit formulas for the Galois action from \cite[Theorem 1.1]{Baction}, \cite[Theorem~1.1]{WINF:birs}.  
\end{remark}

\begin{remark}
The combinatorial description of $\Delta \in \rH_1(U) \wedge \rH_1(U)$ can be related with the ring of cliques as follows.
Consider the graph whose vertices are indexed by the $(n-1)^2$ elements $[E_{i,j}]$ of the 
basis of $\rH_1(U)$.
Place these vertices on $n-1$ levels indexed by the value of $j-i \bmod n-1 \in \{0, \ldots, n-2\}$.
Elements $[E_{i_1,j_1}] \wedge [E_{i_2,j_2}]$ of $\rH_1(U) \wedge \rH_1(U)$ can be indexed by  
a subset of edges in this graph.
The elements in $\Delta$ 
yield the complete graph $K_{n-1}$ on each level; also 
each vertex on level $i$ is connected to $n-2$ vertices from levels $i-1 \bmod n-1$ and $i+1 \bmod n-1$.
\end{remark}

\subsection{Examples}

In Sections~\ref{Sex3}-\ref{Sex5}, we illustrate the process of finding $\Delta$ when $n=3,4,5$;
of course, the results match the formula for $\Delta$ found in Theorem \ref{Tintro}.

\subsubsection{The case $n=3$} \label{Sex3}

Let $\zeta=e^{2 \pi I/3}$.  By Lemma \ref{Lorderloop}:
\begin{eqnarray*}
\tilde{S}_0 & = &  (0,1) \mapsto (1,0) \mapsto (0,\zeta) \mapsto (\zeta,0) \mapsto (0, \zeta^2) \mapsto (\zeta^2, 0) \mapsto (0,1)\\
& = & e_{0,0} \circ e_{0,1}^{-1} \circ  e_{1,1} \circ e_{1,2}^{-1} \circ  e_{2,2} \circ e_{2,0}^{-1};
\end{eqnarray*}
\begin{eqnarray*}
\tilde{S}_1 & = &  (0,1) \mapsto (\zeta^2,0) \mapsto (0,\zeta) \mapsto (1,0) \mapsto (0, \zeta^2) \mapsto (\zeta, 0) \mapsto (0,1)\\
& = & e_{2,0} \circ e_{2,1}^{-1} \circ  e_{0,1} \circ e_{0,2}^{-1} \circ  e_{1,2} \circ e_{1,0}^{-1}; \ {\rm and}
\end{eqnarray*}
\begin{eqnarray*}
\tilde{S}_2 & = &  (0,1) \mapsto (\zeta,0) \mapsto (0,\zeta) \mapsto (\zeta^2,0) \mapsto (0, \zeta^2) \mapsto (1, 0) \mapsto (0,1)\\
& = & e_{1,0} \circ e_{1,1}^{-1} \circ  e_{2,1} \circ e_{2,2}^{-1} \circ  e_{0,2} \circ e_{0,0}^{-1}.
\end{eqnarray*}

By Lemma \ref{LetoE}, 
\begin{eqnarray*}
S^* & = & E_{0,0} \circ E_{0,1}^{-1} \circ  E_{1,1} \circ E_{1,2}^{-1} \circ  E_{2,2} \circ E_{2,0}^{-1} \circ  \\
& & E_{2,0} \circ E_{2,1}^{-1} \circ  E_{0,1} \circ E_{0,2}^{-1} \circ  E_{1,2} \circ E_{1,0}^{-1} \circ \\ 
& & E_{1,0} \circ E_{1,1}^{-1} \circ  E_{2,1} \circ E_{2,2}^{-1} \circ  E_{0,2} \circ E_{0,0}^{-1} \\ 
& = & 
E_{1,1} \circ E_{1,2}^{-1} \circ  E_{2,2}  \circ  
E_{2,1}^{-1} \circ E_{1,2} \circ   
E_{1,1}^{-1} \circ  E_{2,1} \circ E_{2,2}^{-1}. 
\end{eqnarray*}

By Lemma \ref{Lsimplify}(1), in $[\pi_1(U)]_2/[\pi_1(U)]_3$, the image of $S^*$ is the same as the image of
\[(E_{1,1}^{-1} \circ  E_{2,1} \circ E_{2,2}^{-1} \circ E_{1,1}) \circ (E_{1,2}^{-1} \circ  E_{2,2} \circ E_{2,1}^{-1} \circ E_{1,2}).\]

By Lemma \ref{Lsimplify}(2), 
\begin{eqnarray*}
\Delta & = & E_{1,1} \wedge (E_{2,2} - E_{2,1}) + E_{1,2} \wedge (E_{2,1} - E_{2,2})\\
& = &  E_{1,1} \wedge E_{2,2} - E_{1,1} \wedge E_{2,1} +  E_{1,2} \wedge E_{2,1} - E_{1,2} \wedge E_{2,2}.
\end{eqnarray*}


\subsubsection{The case $n=4$}

Let $\zeta=e^{2\pi I/4}$.  By Lemma \ref{Lorderloop}:
\begin{eqnarray*}
\tilde{S}_0 & = &  (0,1) \mapsto (1,0) \mapsto (0,\zeta) \mapsto (\zeta,0) \mapsto (0, \zeta^2) \mapsto (\zeta^2, 0) \mapsto 
(0, \zeta^3) \mapsto (\zeta^3, 0) \mapsto (0,1)\\
& = & e_{0,0} \circ e_{0,1}^{-1} \circ  e_{1,1} \circ e_{1,2}^{-1} \circ  e_{2,2} \circ e_{2,3}^{-1} \circ e_{3,3} 
\circ e_{3,0}^{-1}; 
\end{eqnarray*}
\begin{eqnarray*}
\tilde{S}_1 & = &  (0,1) \mapsto (\zeta^3,0) \mapsto (0, \zeta) \mapsto (1,0) \mapsto (0, \zeta^2) \mapsto (\zeta,0) \mapsto (0, \zeta^3) \mapsto (\zeta^2,0) \mapsto (0,1) \\
& = & e_{3,0} \circ e_{3,1}^{-1} \circ  e_{0,1} \circ e_{0,2}^{-1} \circ  e_{1,2} \circ e_{1,3}^{-1} \circ e_{2,3} 
\circ e_{2,0}^{-1};
\end{eqnarray*}
\begin{eqnarray*}
\tilde{S}_2 & = &  (0,1) \mapsto (\zeta^2,0) \mapsto (0, \zeta) \mapsto (\zeta^3,0) \mapsto (0, \zeta^2) \mapsto (1,0) \mapsto (0, \zeta^3) \mapsto (\zeta,0) \mapsto (0,1) \\
& = & e_{2,0} \circ e_{2,1}^{-1} \circ  e_{3,1} \circ e_{3,2}^{-1} \circ  e_{0,2} \circ e_{0,3}^{-1} \circ e_{1,3} 
\circ e_{1,0}^{-1}; \ {\rm and}
\end{eqnarray*}
\begin{eqnarray*}
\tilde{S}_3 & = &  (0,1) \mapsto (\zeta,0)  \mapsto (0, \zeta) \mapsto (\zeta^2,0) \mapsto (0, \zeta^2) \mapsto (\zeta^3,0) \mapsto (0, \zeta^3) \mapsto (1,0) \mapsto (0,1) \\
& = & e_{1,0} \circ e_{1,1}^{-1} \circ  e_{2,1} \circ e_{2,2}^{-1} \circ  e_{3,2} \circ e_{3,3}^{-1} \circ e_{0,3} 
 \circ e_{0,0}^{-1}.
\end{eqnarray*}

By Lemma \ref{LetoE}, 
\begin{eqnarray*}
S^* & = & 
 E_{1,1} \circ E_{1,2}^{-1} \circ  E_{2,2} \circ E_{2,3}^{-1} \circ E_{3,3} \circ  
  E_{3,1}^{-1}  \circ  E_{1,2} \circ E_{1,3}^{-1} \circ E_{2,3} \circ \\
&&  E_{2,1}^{-1} \circ  E_{3,1} \circ E_{3,2}^{-1}  \circ E_{1,3} \circ 
  E_{1,1}^{-1} \circ  E_{2,1} \circ E_{2,2}^{-1} \circ  E_{3,2} \circ E_{3,3}^{-1}.  
\end{eqnarray*}  

By Lemma \ref{Lsimplify}(1), in $[\pi_1(U)]_2/[\pi_1(U)]_3$, the image of $S^*$ is the same as the image of
  \begin{eqnarray*}
  &&(E_{1,1}^{-1} \circ  E_{2,1} \circ E_{2,2}^{-1} \circ  E_{3,2} \circ E_{3,3}^{-1}  \circ E_{1,1}) \circ \\
  && (E_{1,2}^{-1} \circ  E_{2,2} \circ E_{2,3}^{-1} \circ E_{3,3}  \circ E_{3,1}^{-1}  \circ  E_{1,2}) \circ \\
  && (E_{1,3}^{-1} \circ E_{2,3} \circ E_{2,1}^{-1} \circ  E_{3,1} \circ E_{3,2}^{-1}  \circ E_{1,3}).
 \end{eqnarray*}

By Lemma \ref{Lsimplify}(2), 
\begin{eqnarray*}
\Delta & = &  E_{1,1} \wedge (E_{2,2} - E_{2,1} +E_{3,3} - E_{3,2} ) +\\
&& E_{1,2} \wedge ( E_{2,3} -E_{2,2}+E_{3,1} -E_{3,3} ) + \\
&& E_{1,3} \wedge ( E_{2,1} -E_{2,3} +E_{3,2} -E_{3,1} )+ \\
&& E_{2,1} \wedge (E_{3,2} -E_{3,1}) + \\
&& E_{2,2} \wedge (E_{3,3} -E_{3,2}) + \\
&& E_{2,3} \wedge (E_{3,1} -E_{3,3}).
\end{eqnarray*}

\subsubsection{The case $n=5$} \label{Sex5}

Let $\zeta=e^{2 \pi I/5}$.  By Lemma \ref{Lorderloop}:
\begin{eqnarray*}
\tilde{S}_0 & = &  (0,1) \mapsto (1,0) \mapsto (0,\zeta) \mapsto (\zeta,0) \mapsto (0, \zeta^2) \mapsto (\zeta^2, 0) \mapsto 
(0, \zeta^3) \mapsto (\zeta^3, 0) \mapsto (0, \zeta^4) \mapsto (\zeta^4, 0) \mapsto (0,1)\\
& = & e_{0,0} \circ e_{0,1}^{-1} \circ  e_{1,1} \circ e_{1,2}^{-1} \circ  e_{2,2} \circ e_{2,3}^{-1} \circ e_{3,3} 
\circ e_{3,4}^{-1} \circ e_{4,4} \circ e_{4,0}^{-1};
\end{eqnarray*}
\begin{eqnarray*}
\tilde{S}_1 & = &  (0,1) \mapsto (\zeta^4,0) \mapsto (0,\zeta) \mapsto (1,0) \mapsto (0, \zeta^2) \mapsto (\zeta, 0) \mapsto 
(0, \zeta^3) \mapsto (\zeta^2, 0) \mapsto (0, \zeta^4) \mapsto (\zeta^3, 0) \mapsto (0,1)\\
& = & e_{4,0} \circ e_{4,1}^{-1} \circ  e_{0,1} \circ e_{0,2}^{-1} \circ  e_{1,2} \circ e_{1,3}^{-1} \circ e_{2,3} 
\circ e_{2,4}^{-1} \circ e_{3,4} \circ e_{3,0}^{-1};
\end{eqnarray*}
\begin{eqnarray*}
\tilde{S}_2 & = &  (0,1) \mapsto (\zeta^3,0) \mapsto (0,\zeta) \mapsto (\zeta^4,0) \mapsto (0, \zeta^2) \mapsto (1, 0) \mapsto 
(0, \zeta^3) \mapsto (\zeta, 0) \mapsto (0, \zeta^4) \mapsto (\zeta^2, 0) \mapsto (0,1)\\
& = & e_{3,0} \circ e_{3,1}^{-1} \circ  e_{4,1} \circ e_{4,2}^{-1} \circ  e_{0,2} \circ e_{0,3}^{-1} \circ e_{1,3} 
\circ e_{1,4}^{-1} \circ e_{2,4} \circ e_{2,0}^{-1};
\end{eqnarray*}
\begin{eqnarray*}
\tilde{S}_3 & = &  (0,1) \mapsto (\zeta^2,0) \mapsto (0,\zeta) \mapsto (\zeta^3,0) \mapsto (0, \zeta^2) \mapsto (\zeta^4, 0) \mapsto 
(0, \zeta^3) \mapsto (1, 0) \mapsto (0, \zeta^4) \mapsto (\zeta, 0) \mapsto (0,1)\\
& = & e_{2,0} \circ e_{2,1}^{-1} \circ  e_{3,1} \circ e_{3,2}^{-1} \circ  e_{4,2} \circ e_{4,3}^{-1} \circ e_{0,3} 
\circ e_{0,4}^{-1} \circ e_{1,4} \circ e_{1,0}^{-1}; \ {\rm and}
\end{eqnarray*}
\begin{eqnarray*}
\tilde{S}_4 & = &  (0,1) \mapsto (\zeta,0) \mapsto (0,\zeta) \mapsto (\zeta^2,0) \mapsto (0, \zeta^2) \mapsto (\zeta^3, 0) \mapsto 
(0, \zeta^3) \mapsto (\zeta^4, 0) \mapsto (0, \zeta^4) \mapsto (1, 0) \mapsto (0,1)\\
& = & e_{1,0} \circ e_{1,1}^{-1} \circ  e_{2,1} \circ e_{2,2}^{-1} \circ  e_{3,2} \circ e_{3,3}^{-1} \circ e_{4,3} 
\circ e_{4,4}^{-1} \circ e_{0,4} \circ e_{0,0}^{-1}.
\end{eqnarray*}

By Lemma \ref{LetoE}, 
\begin{eqnarray*}
S^* & = & 
E_{1,1} \circ E_{1,2}^{-1} \circ  E_{2,2} \circ E_{2,3}^{-1} \circ E_{3,3} 
\circ E_{3,4}^{-1} \circ E_{4,4} \circ \\ 
&& E_{4,1}^{-1} \circ  E_{1,2} \circ E_{1,3}^{-1} \circ E_{2,3} 
\circ E_{2,4}^{-1} \circ E_{3,4} \circ  \\ 
&& E_{3,1}^{-1} \circ  E_{4,1} \circ E_{4,2}^{-1} \circ  E_{1,3} 
\circ E_{1,4}^{-1} \circ E_{2,4} \circ  \\ 
&& E_{2,1}^{-1} \circ  E_{3,1} \circ E_{3,2}^{-1} \circ  E_{4,2} \circ E_{4,3}^{-1} \circ 
E_{1,4} \circ \\ 
&& E_{1,1}^{-1} \circ  E_{2,1} \circ E_{2,2}^{-1} \circ  E_{3,2} \circ E_{3,3}^{-1} \circ E_{4,3} 
\circ E_{4,4}^{-1}.
\end{eqnarray*}

By Lemma \ref{Lsimplify}(1)-(2), in $[\pi_1(U)]_2/[\pi_1(U)]_3$, the image of $S^*$ is 
\begin{eqnarray*}
\Delta & = & E_{1,1} \wedge  (-E_{2,1} + E_{2,2} -  E_{3,2} + E_{3,3} - E_{4,3} + E_{4,4})\\
& & +E_{1,2} \wedge (-E_{2,2} + E_{2,3} - E_{3,3} + E_{3,4} - E_{4,4} + E_{4,1})\\
& & +E_{1,3} \wedge(-E_{2,3} + E_{2,4} - E_{3,4} + E_{3,1} -  E_{4,1} + E_{4,2})\\
&& +E_{1,4} \wedge  (-E_{2,4} + E_{2,1} -  E_{3,1} + E_{3,2} -  E_{4,2} + E_{4,3})\\
&& +E_{2,1} \wedge (-E_{3,1} +E_{3,2}-E_{4,2} +E_{4,3})\\
&& +E_{2,2} \wedge  (-E_{3,2} + E_{3,3} - E_{4,3} + E_{4,4})\\
&& +E_{2,3} \wedge (-E_{3,3} + E_{3,4} -  E_{4,4} + E_{4,1})\\
&& +E_{2,4} \wedge (-E_{3,4} + E_{3,1} - E_{4,1} + E_{4,2})\\
&& +E_{3,1} \wedge (-E_{4,1} +E_{4,2} )\\
&& +E_{3,2} \wedge (-E_{4,2} +E_{4,3} )\\
&& +E_{3,3} \wedge (-E_{4,3} + E_{4,4})\\
&& +E_{3,4} \wedge (E_{4,1}-E_{4,4}).
\end{eqnarray*}

\section{The \'etale homology and action of the absolute Galois group} \label{S5}

Let $K=\QQ(\zeta_n)$.  We consider $X$ and $U$ as curves over $K$.
Let $Y \subset U$ be the set of $2n$ points where $xy=0$.
In this section, we denote the \'etale fundamental group by $\pi_1(U)$, 
the \'etale homology by $\rH_1(U)$, and the relative \'etale homology by $\rH_1(U, Y)$.

\begin{remark} \label{Ridentify}
In previous sections, the homology has coefficients in ${\mathbb Z}$;
the \'etale homology has coefficients in a finite or $\ell$-adic ring.
After choosing an embedding $K \subset \mathbb{C}$ and applying Riemann's Existence Theorem, 
we may identify the profinite completion of $\rH_1(U(\mathbb{C}))$ with the \'etale homology $\rH_1(U)$.
Similarly, we may identify the profinite completion of $\pi_1(U(\mathbb{C}))$ with the
\'etale fundamental group $\pi_1(U)$.

We therefore can consider the elements $a_i, z_i, c_j, T, E_{i,j}$ to be in $\pi_1(U)$
and $\bar{a}_i, \bar{z}_i, \bar{c}_j, [E_{i,j}]$ to be in $\rH_1(U)$.
Similarly, we may consider $\beta, e_{i,j}$ to be in the \'etale fundamental groupoid and $[e_{i,j}]$ to be in 
$\rH_1(U, Y)$.
Likewise, we can consider $\Delta$ to be an element of $\rH_1(U) \wedge \rH_1(U)$ 
and its image $\rho$ to be an element of $\rH_1(X) \wedge \rH_1(X)$.
The results in Sections \ref{S2}-\ref{Smain} about these elements remain true in this context as well.
In particular, Theorem~\ref{Tintro} is true in the context of the \'etale homology.
\end{remark}

Beginning in Section~\ref{Sexplicit}, we use the coefficients $\Z/n\Z$ for the \'etale homology, where 
$n$ is the degree of the Fermat curve $X$.
As in Remark~\ref{Ridentify}, we may identify $\rH_1(U(\mathbb{C}); \Z/n\Z)$ with $\rH_1(U; \Z/n\Z)$.

\subsection{An arithmetic property of the action}

\begin{proposition} \label{Pabstractalg}
If $\sigma \in G_\QQ$, then $\sigma$ acts on $\rho$ via the cyclotomic character:
if $\sigma(\zeta)= \zeta^i$, then $\sigma(\rho) = \zeta^i \rho$.
In particular, if $\sigma \in G_K$, then $\sigma$ acts trivially on $\rho$.
\end{proposition}

\begin{proof}
Recall that $\rho$ is a generator for the image of $\rH_2(X) \to \rH_1(X) \wedge \rH_1(X)$.
The map $\rH_2(X) \to \rH_1(X) \wedge \rH_1(X)$ is $G_\QQ$-equivariant. 
By Poincar\'e duality, $\sigma \in G_\QQ$ acts on $\rH_2(X)$ via the cyclotomic character.
The mod $n$ cyclotomic character is trivial when restricted to $G_K$.
\end{proof}

\subsection{Partial information about the $G_{\QQ}$-action} \label{Sexplicit}

Let $n=p$ be a prime satisfying Vandiver's conjecture.
In this section, we collect some information about the action of $G_\QQ$ on $\rH_1(U, Y; \ZZ/p \ZZ)$.

By \cite[Section 10.5]{Anderson}, the action of $\sigma \in G_K$ on 
the generator $\beta$ for $\rH_1(U,Y; \ZZ/p)$ factors through $Q={\rm Gal}(L/K)$.
For $q \in Q$, in \cite[Theorem 1.1]{Baction}, the authors provide a completely explicit formula
for the element $B_q \in \Lambda_1 \otimes \ZZ/p \ZZ$ such that $q \circ \beta = B_q \beta$.

Here is some partial information about how ${\rm Gal}(K/\QQ)$ acts on $\beta$.
Let $a$ be a primitive root modulo $p$.
Let $\xi_a \in {\rm Gal}(K/\QQ)$ be the automorphism such that $\xi_a(\zeta_p)=\zeta_p^a$.
It generates ${\rm Gal}(K/\QQ) \cong (\ZZ/p\ZZ)^*$.
By \cite[Lemma 2.2]{Baction}, ${\rm Gal}(L/\QQ)$ is a semi-direct product of the form $Q \rtimes (\ZZ/p\ZZ)^*$.
We fix a lifting $(1,\xi_a)$ of $\xi_a$ in ${\rm Gal}(L/\QQ)$ and denote it also by $\xi_a$.

Since $\rH_1(U,Y)$ is stabilized by $G_{\QQ}$, there exists $R_a \in \Lambda_1$ such that $\xi_a(\beta) = R_a \beta$.
Modifying the lifting of $\xi_a$ by $q \in Q$ changes $R_a$ by multiplication by the element $B_q \in \Lambda_1$
from \cite[Theorem 1.1]{Baction}.
By \cite[Theorem~7]{Anderson}, $R_a$ is symmetric, meaning invariant when $\epsilon_0$ and $\epsilon_1$
are switched.
By \cite[Section~9.6]{Anderson}, $R_a-1$ is in the augmentation ideal $\langle y_0 y_1 \rangle$.
This is because $\xi_a(\beta)$ and $\beta$ have the same endpoints
and so $R_a\beta - \beta$ is in $\rH_1(U) = \langle y_0 y_1 \rangle \beta$. 
Also $R_a R_b = R_{ab}$.

Proposition~\ref{Pabstractalg} implies that $\xi_a(\rho)=a\rho$.
To state one more property of $R_a$, we consider the 
permutation action on $\Lambda_1$ given by 
${\rm perm}_a(\epsilon_0^i \epsilon_1^j) = \epsilon_0^{ai} \epsilon_1^{aj}$.

\begin{lemma} \label{propertyR}
Let $p$ be an odd prime and let $a$ be a generator of $(\ZZ/p\ZZ)^*$.
Then $\prod_{i=0}^{(p-1)/2 - 1} {\rm perm}_a^i(R_a) = 1$.
\end{lemma}

\begin{proof}
The automorphism $\xi_a^{(p-1)/2}$ is the restriction of complex conjugation to $K$.
This fixes $\beta$, since $\beta$ is defined over ${\mathbb R}$.
By induction, we check that $\xi_a^j(\beta) = (\prod_{i=0}^{j-1} {\rm perm}_a^i(R_a)) \beta$.
Thus \[\beta = \xi_a^{(p-1)/2} (\beta) = \left(\prod_{i=0}^{(p-1)/2-1} {\rm perm}_a^i(R_a)\right) \beta.\]
\end{proof}

If $p=5$, the properties above determine the action of ${\rm Gal}(K/\QQ)$ on $\rH_1(X; \ZZ/p \ZZ)$; 
see Section~\ref{Sverify}.

\section{Examples} \label{Sexample}
\label{examples}

In Section~\ref{Sverify}, if $n=5$, we verify the formula in Theorem~\ref{Tmainresulthere}
through an independent method using invariance properties.
This method provides additional information that allows us to determine the action of $G_\QQ$ on $\rH_1(X)$ if $n=5$; see Section \ref{wholeaction}.
In Section~\ref{Sapplication}, as a final application of the formula if $n=5$, we compute the 
dimension of the $G_K$-invariant subspace of $[\pi]_2/[\pi]_3 \otimes \ZZ/5 \ZZ$ 
and use it to show a coboundary map is trivial. 
For the calculations in this section, we use Magma \cite{Magma}; the code for our calculations is available here 
\cite{Mourcode}.

\subsection{An independent verification of the formula for $\rho$ if $n=5$}  \label{Sverify}

Recall that $\rho$ is a generator of the image of $\rH_2(X) \to \rH_1(X) \wedge \rH_1(X)$.
We study the subspace $\mathcal{A}$ of $\rH_1(X) \wedge \rH_1(X)$ 
of elements that are invariant under ${\rm Aut}(X)$ and $G_K$.
By Propositions~\ref{Pabstractinv} and \ref{Pabstractalg}, $\rho$ is 
contained in $\mathcal{A}$.
Using the material in Section~\ref{Sexplicit}, we determine which elements of $\mathcal{A}$
may be compatible with the action of ${\rm Gal}(K/\QQ)$. 
In Proposition~\ref{Pfinalex},
if $n=5$, we verify that there is a unique $1$-dimensional subspace of $\rH_1(X) \wedge \rH_1(X)$ 
determined by the requirements from the actions of ${\rm Aut}(X)$, $G_K$, and ${\rm Gal}(K/\QQ)$, and 
we verify that this subspace is the same as 
the one given by the formula in Theorem \ref{Tintro}.

\begin{definition} \label{DdefA}
Let $\mathcal{A}$ be the subset of $\alpha \in \rH_1(X) \wedge \rH_1(X)$ that satisfy these properties:
\begin{enumerate}
\item $\alpha$ is invariant under the automorphisms $\phi_0, \phi_1, \tau, \omega$ of $X$; and
\item $\alpha$ is invariant under the action of $\sigma \in G_K$.
\end{enumerate}
\end{definition}

\begin{lemma}
If $n = 5$, then
$\mathcal{A}$ is a 2-dimensional subspace of $\rH_1(X) \wedge \rH_1(X)$. 
\end{lemma}

\begin{proof}
To find $\mathcal{A}$,
we first compute the actions of $\epsilon_0, \epsilon_1, \tau, \sigma$ on $\rH_1(U)$.
Using the exterior wedge product, we then
compute their actions on $\rH_1(U) \wedge \rH_1(U)$. 
Lemma~\ref{Lrelationgamma} provides a basis for the kernel $S$ of $\rH_1(U) \to \rH_1(X)$.
By Lemma~\ref{Lwedge}, $S \wedge \rH_1(U)$ is
the kernel of $\rH_1(U) \wedge \rH_1(U) \to \rH_1(X) \wedge \rH_1(X)$.
We find the image in $\rH_1(X) \wedge \rH_1(X)$ of all 
$D \in \rH_1(U) \wedge \rH_1(U)$ that satisfy these properties:
$\epz D-D$, $\epo D-D$, and $\tau D - D$ are in $S \wedge \rH_1(U)$;
if $\sigma \in G_K$, then $(\sigma-1) D \in S \wedge \rH_1(U)$.  

For the $3$-cycle $\omega \in {\rm Aut}(X)$, it is more complicated to determine the action of $\omega$
on $\rH_1(X)$ since $\omega$ does not stabilize $\rH_1(U)$.
To check invariance under $\omega$, we use a basis for $\rH_1(X)$ found in \cite[Theorem~1.2]{E:integral}, 
together with information about how $\omega$ acts on $\rH_1(X)$ found in \cite[Section~4.3 and Proposition~5.1]{E:integral}.

If $n$ is a prime $p$ satisfying Vandiver's conjecture, the action of $\sigma \in G_K$
on $\rH_1(U)$ can be calculated.  
As explained in the introduction, the reason is that the action of $\sigma$ factors through $Q={\rm Gal}(L/K)$.
In \cite[Theorem~1.1 and Example 3.8]{Baction}, we gave an explicit formula for the action of each $q \in Q$ on 
$\rH_1(U)$.  
This yields an explicit formula for the action of $q \in Q$ on 
$\rH_1(U) \wedge \rH_1(U)$; for $n=5$, we implemented this formula in Magma \cite{Mourcode}.
See Example~\ref{EGQinv} for more details about this.

If $n = 5$, we explicitly find all of the actions above and compute in Magma that
$\mathcal{A}$ is a 2-dimensional subspace of $\rH_1(X) \wedge \rH_1(X)$. 
\end{proof}

By Proposition~\ref{Pabstractalg}, the action of ${\rm Gal}(K/\QQ)$ on $\rho$ is compatible with the cyclotomic character.
We consider which $\alpha \in \mathcal{A}$ have this compatibility property.

Let $a$ be a primitive root modulo $n=p$.
Let $\xi_a$ denote the automorphism $(1,\xi_a) \in {\rm Gal}(L/\QQ)$ from Section~\ref{Sexplicit}.
As seen in Section~\ref{Sexplicit}, the choice of lifting does not matter when working with $G_K$-invariant elements.
Recall that $\xi_a(\zeta) = \zeta^a$.
The element $\alpha \in \mathcal{A}$ is compatible with the cyclotomic character if it is the 
image of an element $D \in \rH_1(U) \wedge \rH_1(U)$ such that
\begin{equation} \label{ElastpropR}
\xi_a(D) - a D \in S \wedge \rH_1(U).
\end{equation}

As seen in Section~\ref{Sexplicit}, $\xi_a(\beta) = R_a \beta$ for some $R_a \in \Lambda_1$ such that:\\
(i) $R_a-1$ is in the augmentation ideal $\langle y_0 y_1 \rangle$;\\
(ii) $R_a$ is symmetric; and\\
(iii) $\prod_{i=0}^{(p-1)/2 - 1} {\rm perm}_a^i(R_a) = 1$ (Lemma~\ref{propertyR}).\\
For $p > 3$, properties (i)-(iii) do not determine $R_a$ but they do give partial information.

\begin{definition}
Let $\mathcal{R}$ be the set of $R_a \in \Lambda_1$ satisfying conditions (i)-(iii). 
\end{definition}

We compute the following in Magma.

\begin{lemma}
If $n=5$, then $\mathcal{R}$ is a set of size 125. 
\end{lemma}

If $n=5$ and $a=2$, the next result shows that
we can uniquely determine
$\langle \rho \rangle$ from these restrictions, despite the ambiguity for $R_a$.

\begin{proposition}\label{Pfinalex}
Let $n=5$ and $a=2$.  
There are exactly $5$ elements $\alpha \in \mathcal{A}$ lying under some
$D \in \rH_1(U) \wedge \rH_1(U)$ 
for which there is an $R_a \in \mathcal{R}$ such that the pair $(D, R_a)$ satisfies \eqref{ElastpropR}. 
These $\alpha$ are exactly the multiples of 
the image in $\rH_1(X) \wedge \rH_1(X)$
of the element $\Delta \in \rH_1(U) \wedge \rH_1(U)$ found in Theorem~\ref{Tintro}.
\end{proposition}

\begin{proof}
This follows from a Magma computation.  We consider all $D \in \rH_1(U) \wedge \rH_1(U)$ lying 
above $\mathcal{A}$.
To compute $\xi_a$ on $D$, we write $D$ as a sum of 
simple tensors $D = \sum_{t \in \mathcal{T}} D_t' \beta \wedge D_t'' \beta$,
where $\mathcal{T}$ is a finite index set and $D_t', D_t'' \in \Lambda_1$.
We compute $\xi_a(D)=\sum_{t \in \mathcal{T}} {\rm perm}_a(D_t') R_a \beta \wedge {\rm perm}_a(D_t'') R_a\beta$.
\end{proof}

We do not know if the analogue of Proposition~\ref{Pfinalex} is true for a prime $n>5$.

\subsection{The action of $G_\QQ$} \label{wholeaction}

Furthermore, if $n=5$ and $a=2$, we have enough information about $R_a \in \mathcal{R}$ to determine the action of 
${\rm Gal}(K/\QQ)$ on $\rH_1(X)$. 

\begin{proposition} \label{PfindRa}
Let $n=5$ and $a=2$.
There are 25 possibilities for $R_a \in \mathcal{R}$ from the calculation in Proposition~\ref{Pfinalex}. Each of the 25 elements $R_a-1$ has the same action on $\rH_1(X)$.  
\end{proposition}

\begin{proof}
Magma calculation.
\end{proof}

Here is one of the possibilities for $R_a$:
\begin{eqnarray*} \label{EfirstR}
R_{a,0} & = & 4\epz^4\epo^3  + 4\epz^4\epo^2 + 2\epz^4\epo  + 3\epz^3\epo^2 + 4\epz^3\epo + 4\epz^2\epo + 4\epz^3 + 4\epz^2   \\
 & +  &  4\epz^3\epo^4 + 4\epz^2\epo^4 + 2\epz\epo^4 + 3\epz^2\epo^3 +  4\epz\epo^3   + 
    4\epz\epo^2 + 4\epo^3 + 4\epo^2 + 3.
    \end{eqnarray*}

Write $y_0=\epz-1$ and $y_1=\epo-1$.  Then
\begin{eqnarray*} \label{EfirstRy}
R_{a,0} & = & 4y_0^4y_1^3 + y_0^4y_1^2 + 2y_0^4y_1 + y_0^3y_1^2 + 4y_0^3y_1  + 4y_0^2y_1 \\ 
& + & 4y_0^3y_1^4 + y_0^2y_1^4  + 2y_0y_1^4 + y_0^2y_1^3  + 4y_0y_1^3 + 4y_0y_1^2 + 2y_0^3y_1^3 + 2y_0y_1 + 1.
    \end{eqnarray*}
The set of $25$ possibilities for $R_a$ in Proposition \ref{PfindRa} is
$\left\{R_{a,0} + i v_1 + j v_2 \mid i,j \in \left\{0, \ldots, 4 \right\} \right\}$,
where: 
\begin{eqnarray*} 
v_1 & = & 2y_0^4y_1^4 + 3y_0^4y_1^2 + 3y_0^3y_1^3 + 3y_0^2y_1^4, \\
v_2 & = & 2y_0^4y_1^4 + 3y_0^4y_1^3 + 2y_0^4y_1^2 + 3y_0^3y_1^4 + 2y_0^2y_1^4.
\end{eqnarray*}
    
Recall from Section \ref{Sproperties} that $R_a$ is well-defined after 
making a choice of automorphism $\xi_a$ in ${\rm Gal}(L/\QQ)$ lifting the automorphism 
$\xi_a \in {\rm Gal}(K/\QQ)$.
Changing $\xi_a$ by $q \in Q$ changes $R_a$ by multiplication by the element $B_q \in \Lambda_1$
found in \cite[Theorem 1.1]{Baction}.

Suppose $\delta \in \rH_1(X)^{G_K}$.
Then $\delta$ is fixed by any automorphism $q \in {\rm Gal}(L/K)$.
By definition, the action of $q$ on $\delta$ is given by multiplication by $B_q$.
Then $B_q R_a \delta = R_a B_q \delta = R_a \delta$ for any $q \in  {\rm Gal}(L/K)$.
This means that the action of ${\rm Gal}(L/\QQ)$ on $\rH_1(X)^{G_K}$
does not depend on the choice involved in the definition for $R_a$.  

Let $J_5$ be the Jacobian of the Fermat curve of degree $5$.
Since $\rH_1(X, \ZZ/5 \ZZ)^{G_L} \cong J_5(L)[5]$ for a number field $L$, 
the next example can be deduced from earlier work of Rorhlich and Tzermias.
Let $J_5^\infty$ be the subgroup of $J_5$ of divisors of degree $0$ supported at the points where
$xyz=0$.
By \cite[Theorem 1]{rohrlich77}, ${\rm dim}_{\ZZ/5 \ZZ}(J_5^\infty)=8$.
By \cite[Proposition, Corollary 2, page 663]{tzermias5}, 
$J_5(\QQ(\zeta_5)) = J_5^\infty$ and
${\rm dim}_{\ZZ/5 \ZZ}(J_5(\QQ))=2$.


\begin{example} \label{EGQinv}
If $n=5$, the $G_K$-invariant subspace of $\rH_1(X; \ZZ/5 \ZZ)$ has dimension $8$, and
the $G_{\Q}$-invariant subspace of $\rH_1(X; \ZZ/5 \ZZ)$ has dimension $2$.
\end{example}

\begin{proof}
Write $n=p$.
The action of $G_K$ on $\rH_1(X; \ZZ/p \ZZ)$ factors through the field extension $L/K$ where 
$L$ is the splitting field of $1-(1-x^p)^p$.
If $p = 5$, there are 3 generators $\tau_0, \tau_1, \tau_2$ for $Q={\rm Gal}(L/K)$.
The formula for the action of each of these on $\rH_1(U; \ZZ/5 \ZZ)$ can be found in \cite[Example 3.8]{Baction}. 

Let ${\rm Fix}(\epz \epo) = \{\alpha \in \rH_1(U; \Z/5 \Z) \mid \epz \epo \alpha = \alpha\}$.

By \cite[Proposition~6.3]{WINF:birs}, letting $S={\rm Fix}(\epz \epo)$,
\footnote{In \cite[Proposition 6.3]{WINF:birs}, we used the notation ${\rm Stab}(\epz \epo)$ instead.}
\begin{equation} \label{Eh1X}
\rH_1(X; \ZZ/5 \ZZ)=\rH_1(U; \ZZ/5 \ZZ)/S.
\end{equation}

In Magma, we computed 
the action of $\tau_0, \tau_1, \tau_2$ on $\rH_1(X; \ZZ/5 \ZZ)$.
To determine the $G_K$-invariant subspace $I$ of $\rH_1(X; \ZZ/5 \ZZ)$, we computed the 
intersection of the kernels of the 3 operators $\tau_i -1$ for $i=0,1,2$.
For the $G_{\Q}$-invariant subspace, we computed the subspace of $I$ 
which is fixed by multiplication by ${\rm perm}_a(R_a)$.
\end{proof}

\subsection{An application about coboundaries} \label{Sapplication} 

Let $p$ be a prime satisfying Vandiver's Conjecture and let $K=\QQ(\zeta_p)$.
In \cite[Theorem 1.1]{Baction}, we gave an explicit formula for the action of $G_K$ on 
$\rH_1(X; \ZZ/p \Z) = \pi/[\pi]_2 \otimes \ZZ/p \Z$.
From the results in this paper, we obtain an explicit action of $G_K$ on the higher quotients 
$[\pi]_m/[\pi]_{m+1} \otimes \ZZ/p \Z$ as well.

We would like to thank the referee for bringing this idea to our attention.
Consider the short exact sequence
\[0 \to (\ZZ/p\ZZ)\rho \to \rH_1(X; \ZZ/p\ZZ) \wedge \rH_1(X; \ZZ/p\ZZ) \to [\pi]_2/[\pi]_3 \otimes \ZZ/p \ZZ \to 0.\]
Since $Q$ fixes $\rho$, this yields a long exact sequence
\begin{equation} \label{Elongexactsec5}
0 \to (\ZZ/p\ZZ) \rho \to \rH^0(Q; \rH_1(X; \ZZ/p\ZZ) \wedge \rH_1(X; \ZZ/p\ZZ)) 
\to \rH^0(Q; [\pi]_2/[\pi]_3 \otimes \ZZ/p\ZZ) \stackrel{\delta}{\to} \rH^1(Q; (\ZZ/p\ZZ) \rho).
\end{equation}

The fact that $Q$ fixes $\rho$ also implies that 
$\rH^1(Q; (\ZZ/p\ZZ) \rho) \cong {\rm Hom}(Q, \ZZ/p\ZZ) \cong (\ZZ/p\ZZ)^{(p+1)/2}$.
Given a $Q$-invariant element $\alpha$ of $[\pi]_2/[\pi]_3 \otimes \ZZ/p\ZZ$, 
consider a lift of $\alpha$ to $\tilde{\alpha} \in \rH_1(X; \ZZ/p\ZZ) \wedge \rH_1(X; \ZZ/p\ZZ)$.
If $q \in Q$, then $q(\tilde{\alpha}) = \tilde{\alpha} + s_q \rho$ for some $s_q \in \ZZ/p\ZZ$.
Then $\delta(\alpha)$ can be identified with the homomorphism $Q \to \ZZ/p\ZZ$ given by $q \mapsto s_q$.

Recall that $X$ has genus $g=(p-1)(p-2)/2$ and 
so $\rH_1(X; \ZZ/p \ZZ) \wedge \rH_1(X; \ZZ/p \ZZ)$ has dimension $\binom{2g}{2}$.
Thus $[\pi]_2/[\pi]_3 \otimes \ZZ/p \ZZ$ has dimension $\binom{2g}{2} -1$. 
If $p=5$, then $g=6$ and $[\pi]_2/[\pi]_3 \otimes \ZZ/5 \ZZ$ has dimension $65$.

\begin{example} \label{Aqinv}
If $p=5$, then
the $G_K$-invariant subspace of $\rH_1(X; \ZZ/5\ZZ) \wedge \rH_1(X; \ZZ/5\ZZ)$ has dimension $35$;
the $G_K$-invariant subspace of $[\pi]_2/[\pi]_3 \otimes \ZZ/5 \ZZ$ has dimension $34$;
and thus the coboundary map $\delta$ in \eqref{Elongexactsec5} is trivial.
\end{example}

\begin{proof}
From the computation in Example \ref{EGQinv}, we know the action of $\tau_i$ on $\rH_1(X; \ZZ/5\ZZ)$ for 
$i=0,1,2$.
From this, we computed the action of $\tau_i$ on $\rH_1(X; \ZZ/5\ZZ) \wedge \rH_1(X; \ZZ/5\ZZ)$
(resp.\ on the quotient of this by $\rho$).
We then computed the dimension of the intersection of the kernels of the 3 operators $\tau_i -1$ for $i=0,1,2$.
This dimension, which is $35$ (resp.\ $34$), is the dimension of the $G_K$-invariant subspace.

The fact that the coboundary map is trivial follows from the exact sequence in \eqref{Elongexactsec5}.
As an additional check (not included here), we computed the cocycle computationally and verified that it is trivial.
\end{proof}

\bibliographystyle{amsalpha}
\bibliography{biblio}

\end{document}